\documentclass[11pt,fleqn]{amsart} % \documentclass[11pt]{article}

\usepackage[usenames,dvipsnames]{xcolor}

\usepackage{amsthm}
\usepackage{amsfonts}
\usepackage[english]{babel}
\usepackage[usenames]{xcolor}
\usepackage{graphicx}
\usepackage{soul}
\usepackage{stfloats}
\usepackage{morefloats}
\usepackage{cite}
\usepackage{lscape}
\usepackage{epstopdf}
\usepackage{braket}
\usepackage[lite]{amsrefs}
\usepackage{mathbbol}
\usepackage{tikz}
\usepackage{tikz-cd}
\usepackage{mathtools}

\usepackage{amsmath,amstext,amsopn,amsfonts,eucal,amssymb}
\usepackage{graphicx,wrapfig,url}

\newcommand{\B}{\mathbb{B}}
\newcommand\N{{\mathbb N}}
\newcommand\Z{{\mathbb Z}}

\newcommand\C{{\mathbb C}}

\newcommand\rank{\mathop{\text{rank}}}

\newtheorem{theorem}{Theorem}[section]
\newtheorem{corollary}[theorem]{Corollary}
\newtheorem{lemma}[theorem]{Lemma}
\newtheorem{proposition}[theorem]{Proposition}
\newtheorem{definition}[theorem]{Definition}

\newtheorem{example}[theorem]{Example}

\newtheorem{remark}[theorem]{Remark}

\newcommand{\spec}{{\rm spec}}

\usepackage{soul}
\begin{document}
\dedicatory{Dedicated to the memory of Professor J. Scott Carter with deepest gratitude }

\title{Surface Braid Invariants from Yang-Baxter Operators
 and Their Cohomology   } 

\author{Masahico Saito} 
\address{Department of Mathematics and Statistics, 
	University of South Florida, Tampa, FL 33620} 
\email{saito@usf.edu} 

\author{Emanuele Zappala} 
\address{Department of Mathematics and Statistics, Idaho State University\\
	Physical Science Complex |  921 S. 8th Ave., Stop 8085 | Pocatello, ID 83209} 
\email{emanuelezappala@isu.edu}

\maketitle

\begin{abstract}
Surface braids are a generalization of braids to surfaces in the 4-disk.  
They are described by planar graphs called charts, and  sequences of braid words called braid systems,  that represent monodromies of branch points in projections. 
Yang-Baxter operators (YBOs) have been used extensively in knot theory for constructions of Jones and other polynomials, defined by braid group representations. 
We propose to use YBOs for constructing surface braid invariants through braid systems. 
The intersections of eigenspaces are used for this construction. 
We further use 3-cocycles of the Yang-Baxter cohomology theories via charts, in an analogue of quandle cocycle invariants but more generally utilizing  tensor products of modules. 
An approach towards defining knotted surface  invariants through closures of surface braids is discussed.

\bigskip


\end{abstract}

\date{\empty}

\tableofcontents

\section{Introduction}

Braid theory has played a significant role in a wide range of areas in mathematics and theoretical physics.
In knot theory, the Jones polynomial~\cite{Jones} was discovered through braid group representations, and it later led to the area of quantum topology. A family of braid group representations has been used in physics, as solutions to the Yang-Baxter equation (YBE), that corresponds to a key braid relation, as well as Reidemeister type III move for knot diagrams. Applications of braid theory to knot theory 
 are based on theorems by Alexander and Markov. The former states that any link can be represented by a closed braid, and the latter provides the equivalence among  
 braids that results in equivalence of links. % equivalence.

Since the discovery of the Jones polynomial, its extensions to higher dimensions have been in question. 
This question led  to categorifications~\cite{Kh}  of Jones and other polynomials.
An invariant in a state sum form, called the quandle cocycle invariant for knots and knotted surfaces, was also motivated by this question, and developed earlier, though the states are restricted by quandle colorings. For computations of the cocycle invariant, the surface braid theory was employed.
It is also natural for this question  to ask if Yang-Baxter operators (YBOs), i.e. solutions of the YBE, can be employed in higher dimensions. Exploring this question is the purpose of this paper.

The surface braids were 
 extensively developed by Kamada~\cite{Kamada:SB}. 
Surface braids are properly embedded compact orientable surfaces in the 4-disk $D^2 \times D^2$ 
such that the restriction of the projection $D^2 \times D^2 \rightarrow D^2$ is a simple branched cover (see next section). Analogues of theorems of Alexander and Markov are established, and their applications to knotted surfaces have been given. As in the classical case, surface braids are of interest in their own right as well.

Surface braids are represented by planar graphs called charts, corresponding to self intersections of projections to 3-space (details in the next section). The image of branch points, called black vertices in the disk, have monodromies expressed by classical braid words. This correspondence gives a sequence of braid words called a braid system. There is a known  equivalence defined on braid systems, called slide equivalence, that realizes the equivalence of surface braids. 
	Therefore, the problem of classifying surface braids is closely related to the classification of 
	braid word sequences 
	 up to slide equivalence, which is undecidable \cite{Lib-Tei}. 
	 As a consequence, the construction of invariants of braid systems up to slide equivalence for the study of surface braids is a relevant and important problem. 
Since surface braids are represented by sequences of braids, taking braid representations induced by YBOs, 
 that are used in defining the Jones and other polynomials, would be a natural approach.  
This is the first method we propose in this paper. The second method is to employ Yang-Baxter  cohomology theory~\cite{Eisermann,SZ-YBH} applied to 
 white vertices, that correspond to triple points in projections.
Specifically,  3-cocycles of the Yang-Baxter cohomology theory 
are assigned to triple points, motivated by quandle 3-cocycle invariants of knotted surfaces. This construction is performed on tensor products of modules, instead of discrete structures, and 
 the resulting invariant   is called  a quantum cocycle invariant. A similar  idea of constructing quantum cocycle invariants for classical knots  using Yang-Baxter 2-cocycles  is found in \cite{SZ-QC}.

\begin{figure}[htb]
\begin{center}
\includegraphics[width=4.5in]{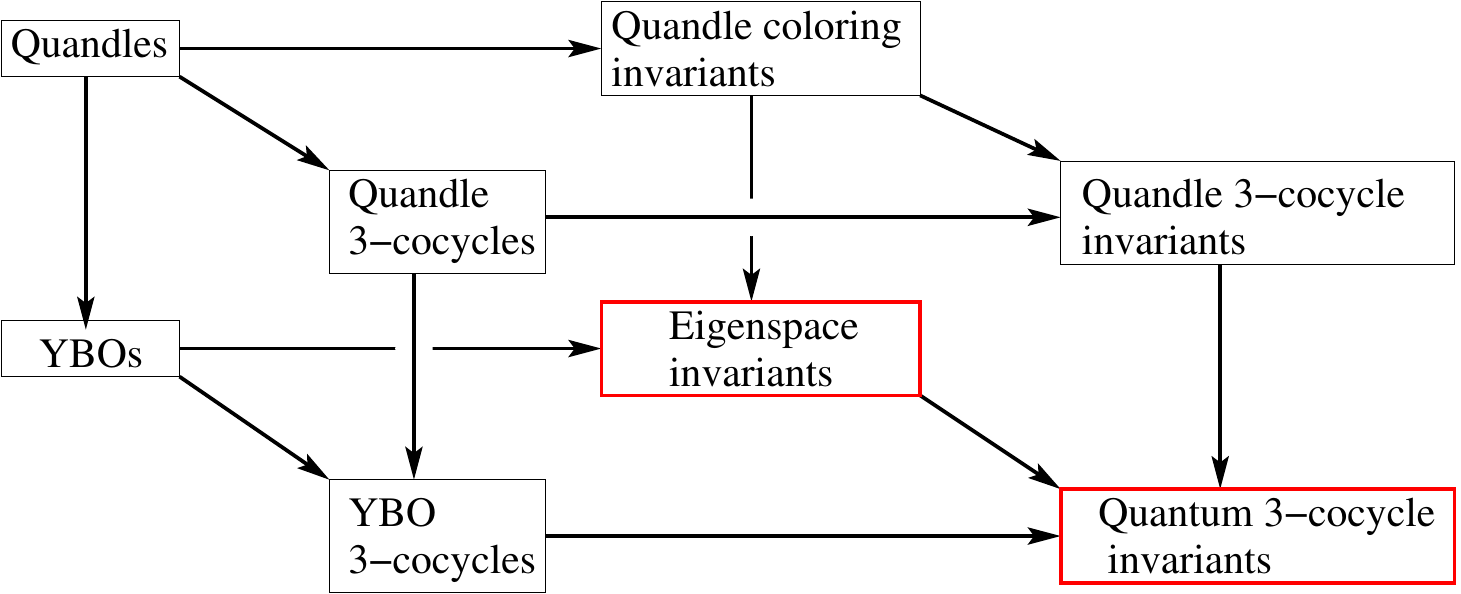}
\end{center}
\caption{Conceptual diagram for constructions of  invariants}
\label{scheme3cocy}
\end{figure}

More specifically, relations among these invariants are schematically depicted in Figure~\ref{scheme3cocy}.
From the left of the figure, YBOs can be constructed from quandles, and there are cohomology theories for quandles and YBOs,
giving 3-cocycles that are used for invariants. Quandles give rise to coloring invariants, and the corresponding invariants for YBOs are eigenspace invariants we construct in this paper. Quandle colorings, together with 3-cocycles provide quandle 3-cocycle invariants for knotted surfaces. The corresponding 3-cocycle invariants  
  for YB cohomology 
are defined for surface braids in this paper. Since YBOs provide classical quantum knot invariants, we call  invariants from YB 3-cocycles the quantum 3-cocycle invariants.

The paper is organized as follows. 
Concepts and notations used in this paper are reviewed in Section~\ref{sec:prelim}.
The main object of study, surface braids is reviewed, and the tools of study, such as  charts, Hurwitz arc systems, braid systems are established. 
 The braid systems correspond to branch points of projection of surface braids, and represented by 1-valent vertex 
called black vertex of charts. 
The main algebraic tool for the study, namely Yang-Baxter operators (YBOs), is also presented. The YBOs have been extensively used in knot theory in dimension 3, and its use in dimension 4 is the main novel aspect of the paper.

In Section~\ref{sec:eigen}, eigenspaces of YBOs are  examined, and their intersections with respect to braid systems
are defined.
 Two types of eigenspaces are presented for the purpose of ease of computations, and their relations are discussed.
The invariance of the intersections is proved in Section~\ref{sec:eigeninvariance}.
This is done by checking the equivalence relation among braid systems called the slide equivalence.

The motivation for this definition of taking the intersections of eigenspaces stems from quandle coloring numbers, that are defined for knotted surfaces as well as classical knots.
Quandle invariants have been extensively studied in both dimensions.
In Section~\ref{sec:quandle_col} the relation between the eigenspace invariant and the quandle coloring invariant is presented. Explicit examples from quandles are discussed.
This shows that the eigenspace invariant can be seen as an approach to define quantum version of quandle invariants.

The YBOs are extensively used in classical knots, and the most fundamental example of YBOs is that of the Kauffman bracket, that defines the Jones polynomial. In Section~\ref{sec:bracket}, the eigenspaces of  such YBO are examined, 
and the intersection invariant is explicitly computed for surface braids corresponding to a family of  spun torus knots.
It is shown that the invariant distinguishes infinitely many of them.

%The braid systems correspond to brach points of projection of surface braids, and represented by 1-valent vertex  called black vertex of charts. 
 In addition to black vertices, there is another type of  important vertices of charts, called a white vertex, that corresponds to triple points in projections. Triple points are used in quandle 3-cocycle invariants of knotted surfaces. 
In Section~\ref{sec:q3} 
we turn our attention to white vertices, and define a quantum version of 3-cocycle invariant, using 3-cocycles of the Yang-Baxter cohomology. 
Well-definedness is proved by chart moves, and relation to quandle invariants are examined.

The Jones polynomial was originally defined from braid group representation and then  by applying theorems of Alexander and Markov. These theorems are known to exist in dimension 4. 
Thus knotted surface invariants can be  defined from surface braid invariants if they are shown to be invariant under conjugation and stabilization. We examine approaches to derive such invariants from the YBO invariants discussed in this paper.

\section{Preliminary}\label{sec:prelim}

A {\it surface braid} of index $n$  is a properly embedded surface  $S$
 in the product of 2-disks $D_1 \times D_2$
such that $\pi: D_1 \times D_2 \rightarrow D_2$ restricts to $S\subset D_1 \times D_2$ as a simple covering map, and $\partial S \subset D_1 \times \partial D_2$ is a trivial braid $X_n \times \partial D_2$, where
$X_n$ is a set of  fixed points in $D_1$. A  simple covering means $\# ( \pi^{-1} (y) )=n$ or $n-1$. 
The singular points correspond to isolated branch points~\cite{Kamada:SB}.

\begin{figure}[htb]
\begin{center}
\includegraphics[width=2.5in]{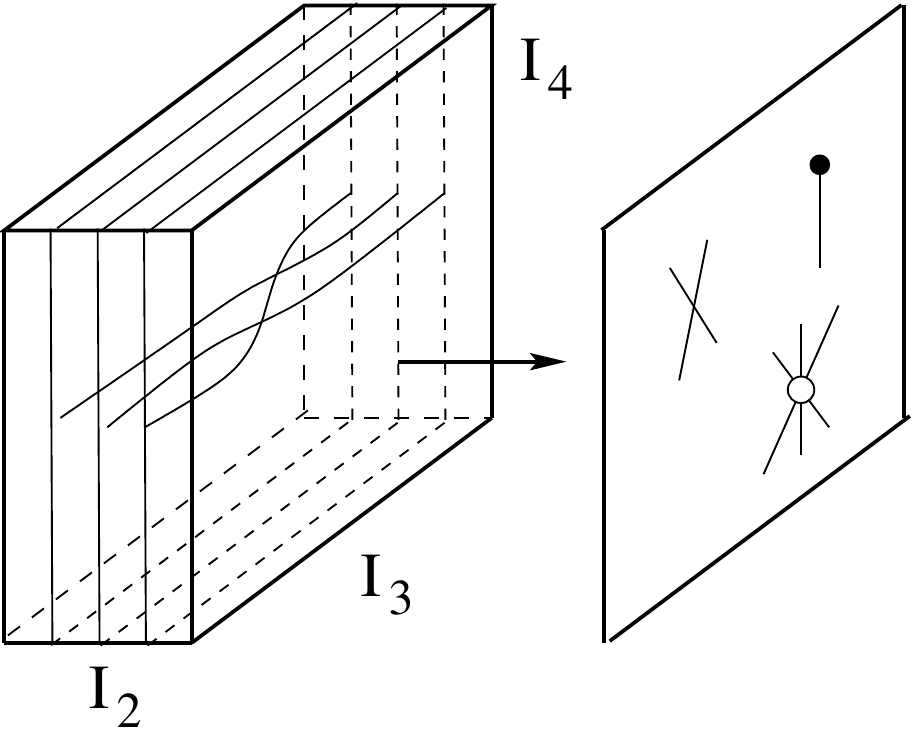}
\end{center}
\caption{Projections of a surface braid}
\label{box}
\end{figure}

Writing $D_1=I_1 \times I_2$ and $D_2=I_3 \times I_4$, 
We may assume that $\pi$ factors through $\pi': D_1 \times D_2 \rightarrow I_2 \times D_2$
and $\pi'$ is a generic projection, so that $\pi'(S)$ consists of embedded points, double point curves, isolated branch and triple points~\cite{Kamada:SB}.
Let  $\pi'': I_2 \times D_2 \rightarrow D_2$ be the projection, so that $\pi =  \pi '' \circ \pi'$. 
The image of the singular points under $\pi''$ 
is called a {\it chart} $\Gamma(S)$ of a surface braid $S$, which is a planar graph having vertices 
of valency 1 (the image of a branch point, called a black vertex), 2 (called a crossing) and 6 (the image of a triple point, called a white vertex).

In Figure~\ref{box} the image $\pi'(S)$ in $I_2 \times I_3 \times I_4$ and the projection $  I_2 \times D_2 \rightarrow  D_2=I_3 \times I_4$ is schematically depicted. Each cross section $ I_2 \times  I_3 \times \{ t \} \cap \pi'(S)$ is identified  wth a classical braid diagram except singular values  of $t$.
Thus the edges of a chart corresponds to standard braid generators $\sigma_j$, and the index $j$ is assigned to each edge.
 The edges are oriented to distinguish  $\sigma_j$ and $\sigma_j^{-1}$, see below.

On the right in Figure~\ref{box},  building blocks of images of singular points are depicted in $D_2$.
More details of neighborhoods are depicted in Figure~\ref{sbraids}, for  a brach point (A), a crossing (B), and a triple point (C), with images in the projection at the right.

\begin{figure}[htb]
\begin{center}
\includegraphics[width=4.5in]{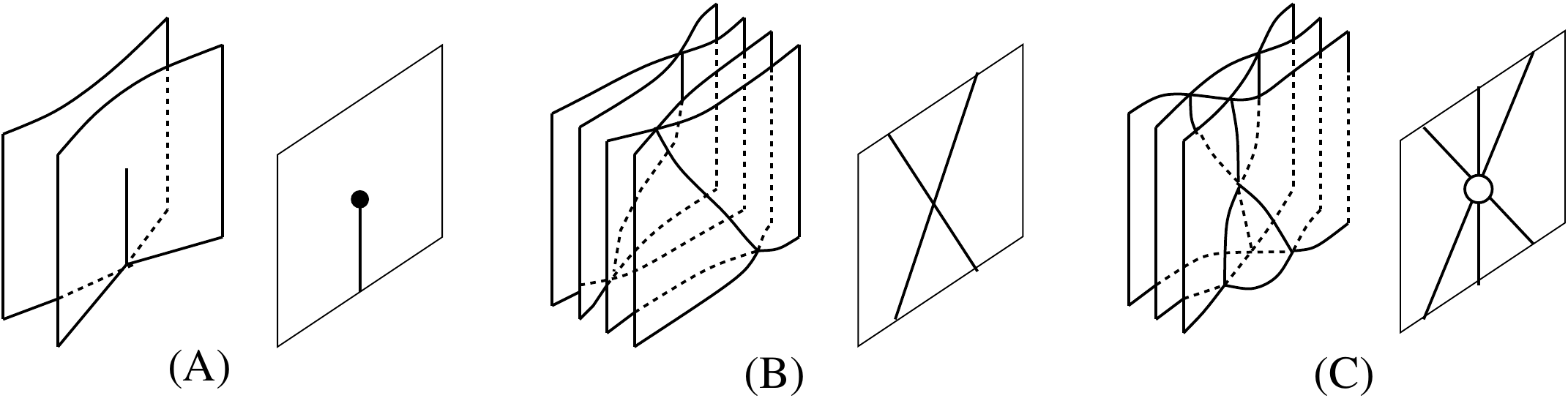}
\end{center}
\caption{Building blocks of a surface braid}
\label{sbraids}
\end{figure}

 Let $\Sigma(S)=(y_1, \ldots, y_m)$ be the set of black vertices of a chart $\Gamma(S)$ of a surface braid $S$.
 Let $a_i$ be an arc from $y_i$ to the base point $q_0$ that are in general position with $\Gamma(S)$. 
 The set of arcs ${\mathcal A}=(a_1, \ldots, a_m)$ is called  a {\it (Hurwitz) arc system} \cite{Kamada:SB} (Section 2.3). See Figure~\ref{Hurwitz} (A).

If a given arc that is in general position with $\Gamma(S)$ crosses an oriented edge labeled by $j$ of $\Gamma(S)$ from the left to right with regard to the orientation of the edge, 
then the crossing corresponds to a standard generator $\sigma_j$, and $\sigma_j^{-1}$ if it crosses from right to left
\cite{Kamada:SB}. %(Figure 18.14). 
See Figure~\ref{Hurwitz} (B).

For a Hurwitz arc $a_i$, we associate a {\it braid monodromy} as follows. 
One goes along $a_i$ from $q_0$ (in opposite direction)
 to nearby $y_i$. Read the sequence of braid generators or their inverses as it crosses the edges of $\Gamma(S)$
 as in Figure~\ref{Hurwitz} (B).
 The sequence is written from the right; if one encounters $\sigma_1$ and $\sigma_2$ in this order, the sequence records $\sigma_2 \sigma_1$. Let $w_i$ be the braid word thus constructed from $q_0$ to nearby $y_i$.
 
Around a  black vertex $y_i$,  a loop goes counterclockwise.
When 
 the orientation of an edge of $\Gamma(S)$ is incoming to $y_{i}$, the corresponding loop around $y_i$
  is $\sigma_{k_i}$, and an outgoing edge corresponds to $\sigma_{k_i}^{-1}$. 
    Then one goes back along $a_i$, which reads $w_i^{-1}$, and obtains the braid word
    $b_i := w_i^{-1} \sigma_{k_i}^{\epsilon( i )} w_i$. 
See Figure~\ref{Hurwitz} (C).
The sequence $\vec{b}=(b_i, \ldots, b_m)$ is called the {\it braid system} of $S$.

\begin{figure}[htb]
\begin{center}
\includegraphics[width=4.2in]{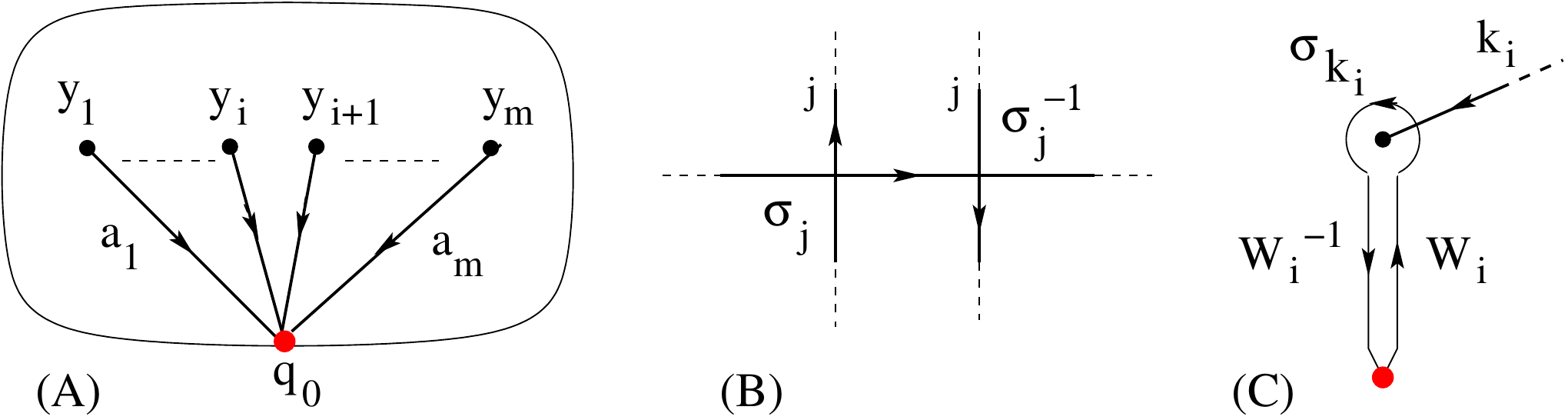}
\end{center}
\caption{Hurwitz arc system and monodromy}
\label{Hurwitz}
\end{figure}

Below we consider braid representations defined from the Yang-Baxter operators (YBOs).
Let $V$ be a $\mathbb k$-module for a unital ring $\mathbb k$. 
A YBO is an automorphism $R: V \otimes V \rightarrow V \otimes V$ satisfying the Yang-Baxter equation (YBE)
$$ (R \otimes \mathbb 1)(\mathbb 1 \otimes R)  (R \otimes \mathbb 1)
= (\mathbb 1 \otimes R)  (R \otimes \mathbb 1)(\mathbb 1 \otimes R) , $$
which corresponds to the braid relation $\sigma_i \sigma_{i+1} \sigma_i = \sigma_{i+1} \sigma_i \sigma_{i+1} $ and the Reidemeister type III move in diagrams.
 For a   set $X$,  a map $R: X \times X \rightarrow X \times X$ is called a set-theoretic Yang-Baxter operator (SYBO) if $R$ satisfies a similar equation with $\otimes$ replaced by $\times$, called the set-theoretic Yang-Baxter equation (SYBE).

For a given YBO $R$ on $V$ we consider the representation $\rho: \mathbb B_n \rightarrow {\rm Hom}(V^{ \otimes n },  V^{ \otimes n } ) $
defined for standard generators by 
$$\sigma_j \mapsto R_{(j)}:= {\mathbb 1}^{\otimes (j-1)} \otimes R \otimes {\mathbb 1}^{\otimes (n-j-1)} .$$
For the inverse $R^{-1}$, we denote  $R_{(j)}^{-1}= {\mathbb 1}^{\otimes (j-1)} \otimes R^{-1} \otimes {\mathbb 1}^{\otimes (n-j-1)} .$
To avoid confusion, throughout the article we mostly use $i$ for subscripts of braid system entries $b_i$, with total number $m$, and $j$ for indices of braid generators $\sigma_j$, with the braid index $n$.

When we consider linear operators  corresponding to a braid word under a representation
$\rho: \B_n \rightarrow {\rm Hom}(V,V)$ for a module $V$, we follow regular matrix convention of left action.
Since the braid words in a braid system are written from the right to left as a sequence of braid generators, 
if the matrices corresponding to $w_i$ and $\sigma_j$ are $W_i$ and $R_{(j)}$, respectively,  then for $b_i= w_i^{-1} \sigma_{k_i}^{\epsilon(i)} w_i$
we have $\rho (b_i)=   W_i^{-1}  \ {R_{(k_i)}}^{\epsilon(i)} \  W_i $
in the same order.

\section{Eigenspaces of braid representations by YBOs}\label{sec:eigen}

Let $\Lambda$ be the set of eigenvalues of a YBO $R$. 
Denote the eigenspace of $R$ for $\lambda \in \Lambda$ by $E_\lambda$.
Recall that the set of eigenvalues of $R^{-1}$ is in bijection with 
$\Lambda$ by $\lambda \leftrightarrow \lambda^{-1}$, so that we denote the set of eigenvalues of $R^{-1}$ by $\Lambda^{-1}$.
Furthermore, the eigenspace of $R^{-1}$ for $\lambda^{-1}$ is equal to $E_\lambda$.
 For $\lambda \in \Lambda$, we denote the eigenspace of $R^{-1}$ for $\lambda^{-1}$ by 
$E_\lambda^{-1}$ to avoid a confusion when $\lambda^{-1}$ is also an eigenvalue of $R$. 

For the $n$-braid group $\mathbb B_n$ and its representation $\rho:  \mathbb B_n \rightarrow {\rm Hom}(V^{ \otimes n },  V^{ \otimes n } ) $, we use the following notations. 
Denote by $E(R_{(j)}, \lambda)$ the eigenspace of $\rho(\sigma_j ) = R_{(j)}= {\mathbb 1}^{\otimes (j-1)} \otimes R \otimes {\mathbb 1}^{\otimes (n-j -1)} $.
 If $R$ has an eigenvalue $\lambda$ and $R^{-1}$ has the corresponding eigenvalue $\lambda^{-1}$, 
 we denote the eigenspace of $R_{(j)}^{-1}$  for $\lambda^{-1}$ by $E(R_{(j)}^{-1}, \lambda^{-1})$.

Let $\mathbb B_n$ denote the $n$-braid group.
Let $S$ be a surface braid with a chart $\Gamma$.
Let 
$$\vec{b}=
(b_1 , \ldots,  b_m)=(w_1^{-1} \sigma^{\epsilon_1}_{k_1} w_1, \ldots, w_m^{-1} \sigma^{\epsilon_n}_{k_m} w_m )$$
be a braid system of a surface braid $S$ with respect to a chart 
 $\Gamma (S)$ 
as in \cite{Kamada:SB} (Section 18.10).

Denote by $\rho(\vec{b})= (\rho (b_1), \ldots, \rho(b_m))$ and
%(\rho ( w_1^{-1} \sigma^{\epsilon_1}_{k_1} w_1) , \ldots, \rho( w_n^{-1} \sigma^{\epsilon_n}_{k_n} w_n)  )$.
set $B_i:= \rho ( b_i ) $ for $i=1, \ldots, m$.
If $b_i=w_i^{-1} \sigma_{k_i}^{\epsilon(i)} w_i$,  we set  $W_i:=\rho(w_i)$. 
For $b_i=w_i^{-1} \sigma_{k_i}^{\epsilon(i)} w_i$, we have 
$$B_i=\rho(b_i)=\rho(w_i^{-1}   \sigma_{k_i}^{\epsilon(i)} w_i)  = \rho(w_i^{-1} ) \rho (\sigma_{k_i}^{\epsilon(i)} ) \rho (w_i) = W_i^{-1}  \ {R_{(k_i)}}^{\epsilon(i)} \ W_i . $$

Denote by $E(B_i, \lambda^{\epsilon(i)} )$ the eigenspace of 
$B_i=W_i^{-1} \ {R_{(k_i)}}^{\epsilon(i)} \ W_i$.
We stress that although $\lambda^{-1}$ appears in the expression, they appear only when the central term is $R_{k_i}^{-1}$ whose eigenvalue is $\lambda^{-1}$, that corresponds to the specified eigenvalue $\lambda$ of $R$. 
In other words, throughout the discussion an eigenvalue $\lambda$ of $R$ is fixed.

\begin{lemma}\label{lem:supset}
If a YBO $R$ has an eigenvalue $\lambda$, then $R_{(j)}= {\mathbb 1}^{\otimes (j-1)} \otimes R \otimes {\mathbb 1}^{\otimes (n-j -1)} $ on $V^{\otimes n}$ 
has an eigenvalue $\lambda$ for all $j=1, \ldots, n$.
Moreover, if $E_\lambda$ is the eigenspace of $\lambda$ for $R$, 
the eigenspace  $E(R_{(j)}, \lambda)$ of $\lambda$ for $R_{(j)}$ contains
\begin{equation}
E_{(j, \lambda)} := V^{\otimes (j-1)} \otimes E_\lambda  \otimes V^{\otimes (n-j-1)}. \label{eqn:Elambda}
\end{equation}
Similarly, for the inverse  the eigenspace $E(R_{(j)}^{-1}, \lambda^{-1})$ of $R_{(j)}^{-1}$  contains 
\begin{equation}
E^{-1}_{(j, \lambda)} := V^{\otimes (j-1)} \otimes E_\lambda^{-1}  \otimes V^{\otimes (n-j-1)}. \label{eqn:Einv}
\end{equation}
That is, $E(R_{(j)}, \lambda ) \supset  E_{(j, \lambda)}$ and 
$E(R_{(j)}^{-1}, \lambda^{-1} ) \supset E^{-1}_{(j, \lambda)}$. 

Moreover, if $V$ is a flat $\mathbb k$-module, and $\lambda$ is an eigenvalue of $R$, then $E(R_{(j)}, \lambda ) =  E_{(j, \lambda)}$ and $E(R_{(j)}^{-1}, \lambda^{-1} ) = E^{-1}_{(j, \lambda)}$. In this situation, if $R$ is diagonalizable, the follwing eigenspace decomposition %would hold:
holds: $V^{\otimes n} = \bigoplus_{\lambda \in \spec(R)} E_{(j,\lambda)}$, where $\spec(R)$ indicates the spectrum of $R$. 
\end{lemma}

\begin{proof}
Let $v$ be an eigenvector of $R$ for $\lambda$, so that $R v=\lambda v$.
Then for any $u \in V^{\otimes (j-1)} $ and $w \in V^{\otimes (n-j-1)} $,
we have $R_{(j)} (u \otimes v \otimes w)=\lambda u \otimes v \otimes w$, hence $\lambda$ is an eigenvalue of 
$R_{(j)}$ with an eigenvector $u \otimes v \otimes w$.
We have that $u \otimes v \otimes w \in E_{(j, \lambda)}$, hence the statement follows.
The case for the inverse is similar.

Let $R$ have an  eigenvalue $\lambda$ over a flat $\mathbb k$-module $V$. The tensor product $V\otimes V$ is flat over $\mathbb k$ as well. It follows that the right and left tensor product functors by $V$ are exact. The eigenspace $E_\lambda$ fits into the shorth exact sequence 
$$
\begin{tikzcd}
	0\arrow[rr] && E_\lambda \arrow[rr] && V\otimes V\arrow[rr,"R-\lambda \mathbb 1^{\otimes 2}"] && V\otimes V. 
\end{tikzcd}
$$
Tensoring on the left, and using the exactness of the tensor product functor, we obtain the short exact sequence
$$
\begin{tikzcd}
0\arrow[rr] && V \otimes E_\lambda \arrow[rr] && V\otimes V\otimes V\arrow[rr,"\mathbb 1\otimes R-\lambda \mathbb 1^{\otimes 3}"] && V \otimes V\otimes V, 
\end{tikzcd}
$$
which shows that $E(\mathbb 1\otimes R,\lambda) = \ker(\mathbb 1\otimes R-\lambda \mathbb 1^{\otimes 3}) = V\otimes \ker(R-\lambda \mathbb 1^{\otimes 2}) = V\otimes E_\lambda$. A similar procedure by tensoring on the right shows that $E(R\otimes \mathbb 1,\lambda) = E_\lambda\otimes V$. Iterating this process shows that more generally $E(R_{(j)},\lambda) = E_{(j,\lambda)}$. An analogous reasoning can be applied to the case of inverses. 

To prove that the decomposition in the special case of $R$ admitting a basis of eigenvectors, that is when $R$ is diagonalizable, observe that the previous part of the proof gives 
\begin{eqnarray*}
		V^{\otimes n} &=& V^{\otimes (j-1)} \otimes (\oplus_{\lambda \in \spec(R)} E_\lambda) \otimes V^{\otimes (n-j-1)}\\
		&=& \bigoplus_{\lambda \in \spec(R)} V^{\otimes (j-1)} \otimes E_\lambda \otimes V^{\otimes (n-j-1)}\\
		&=& \bigoplus_{\lambda \in \spec(R)} E_{(j,\lambda)}
\end{eqnarray*}
%}
as desired.
\end{proof}

For all $i$, recall that $B_i$ is written as 
$W_i^{-1} R_{(k_i)} W_i$, where $W_i=\rho (w_i)$. 
Recall also that  $E(R_{(j)}, \lambda)$  and $E(B_i, \lambda)$ 
 denote the eigenspaces of $R_{(j)}$ and $B_i$ for eigenvalue $\lambda$, respectively.

\begin{lemma}\label{lem:conjeigen}
Assume that  $\lambda$ is an eigenvalue of $R$ with eigenspace $E_\lambda$.
Let  $E(R_{(j)}, \lambda)$
and $E(B_i, \lambda)$ be as above.
Then $\lambda$ is an eigenvalue of $B_i$ with eigenspace $E(B_i, \lambda) = W_i^{ -1} ( E(R_{(k_i)}, \lambda ) )$.

Similarly, we have $E_{(B_i, \lambda)} = W_i^{-1} ( E_{(R_{(k_i)}, \lambda )} )$.
\end{lemma}

\begin{proof}
In general if $E$ is an eigenspace of $X$, and $P$ is  a non-singular matrix, then the eigenspace of $P^{-1} X P$ for the same eigenvalue
is $P^{-1}(E)$.
Hence the statement follows from Lemma~\ref{lem:supset}.
\end{proof}

\begin{definition}
{\rm

For a braid system $\vec{b}=(b_1, \ldots, b_m)$ of a surface braid $S$
and a representation $\rho$ with respect to a YBO $R$, 
denote by $\vec{B}:=(B_1, \ldots, B_m)$.
Define the {\it braid system eigenspace} $E(\vec{B}, \lambda)$ for an eigenvalue $\lambda$ of $R$ by
the intersection $\bigcap_{i=1}^m E(B_i, \lambda)= \bigcap_{i=1}^m  W_i^{-1} ( E(R_{(k_i)}, \lambda )) $. 
Similarly, define $E_{(\vec{B}, \lambda)} := \bigcap_{i=1}^m W_i^{-1} ( E_{(k_i, \lambda)} ) $.

For the empty braid system, we define $\vec B=I$, the identity matrix of rank $n$, 
and $E(\vec B, \lambda)= V^{\otimes m}$ for $\lambda =1$ and $0$ otherwise. 
}
\end{definition}

Lemma~\ref{lem:supset} implies that $E(\vec{B}, \lambda)\supset E_{(\vec{B}, \lambda)} $.
Thus nontriviality of $E_{(\vec{B}, \lambda)} $ implies nontriviality of $E(\vec{B}, \lambda)$.
The reason for having two definitions is that $E_{(\vec{B}, \lambda)} $ is computable while computing $E(\vec{B}, \lambda)$
is elusive for non-flat modules over $\mathbb k$ at this time.
Therefore we focus on $E_{(\vec{B}, \lambda)} := \bigcap_{i=1}^m W_i^{-1} ( E_{(k_i, \lambda)} ) $.

%(Reminder:)  It is known
%(found in mathoverflow in comment in TeX file)
%%https://mathoverflow.net/questions/268581/tensor-products-and-intersections-of-subspaces
%that for subspaces 
%if $A, C \subset V$, $A \cap C \neq \{ 0 \} $,
%$B, D \subset U$, $B \cap D \neq \{ 0 \} $,
%then $(A \otimes B ) \cap (C \otimes D)=(A \cap C) \otimes (B \cap D)$. 

%If $i<j$ and $j-i>1$, consider the first $k_1$ and second $k_2$ factors $V^{\otimes k_1} \otimes V^{\otimes k_2}=V^{\otimes n}$ such that the $R$-factors in $R_{(i)}$ and $R_{(j)}$ are in $$V^{\otimes k_1} $ and $$V^{\otimes k_2} $, respectively. From the above statement we obtain 

\section{Invariance of  braid system eigenspaces under slide equivalence}\label{sec:eigeninvariance}

The goal of this section is to show that the  braid system eigenspace $E=E(\vec{B}, \lambda)$ 
 is invariant under slide equivalence of braid systems, which implies that $E$ is an invariant of surface braids.

The slide equivalence of braid systems \cite{Kamada:SB}, Section 2.7, 
is generated by 
\begin{eqnarray}
	(b_1, \ldots, b_m) &\sim&  (b_1 , \ldots, b_{i-1}, b_i b_{i+1} b_i^{-1} , b_i, b_{i+2}, \ldots, b_m).
	\label{eqn:slide1} %\\
%	(b_1, \ldots, b_n) &\sim&  (b_1 , \ldots, b_{i-1}, b_{i+1}, b^{-1}_{i+1} b_i b_{i+1} , b_{i+2}, \ldots, b_n).\label{eqn:slide2}
\end{eqnarray}

\begin{figure}[htb]
\begin{center}
\includegraphics[width=3in]{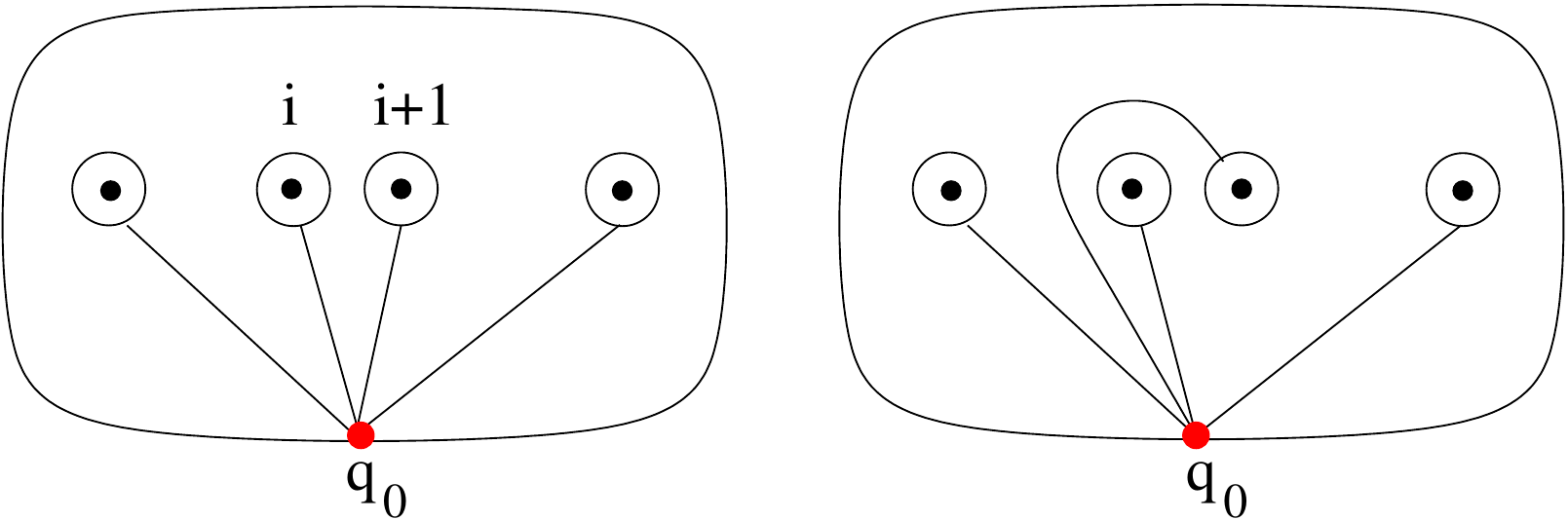}
\end{center}
\caption{Slide equivalence on charts}
\label{slide}
\end{figure}

The slide equivalence move \eqref{eqn:slide1} is depicted in Figure~\ref{slide}, where the move is to change the left figure to the right. The inverse action follows from the above.

\begin{theorem}\label{thm:sfc_braid}
The spaces $E(\vec{B}, \lambda)$ and $E_{(\vec{B}, \lambda)}$ are invariants of surface braids, 
which we call {\emph the Yang-Baxter (YB) eigenspace invariant}. 
Therefore we also write $E(S, \lambda)$ and  $E_{(S, \lambda)}$.
\end{theorem}
\begin{proof}  
		It is known~\cite{Kamada:SB} that equivalent surface braids have slide equivalent braid systems. Hence we show that the eigenspaces are invariant under slide equivalence. 
			Set $\vec{b}=(b_1, \ldots, b_m)$
		and let
		\begin{eqnarray*}
				\vec{b}' = (b_1', \ldots, b_m')= (b_1 , \ldots, b_{i-1}, b_i b_{i+1} b_i^{-1} , b_i, b_{i+2}, \ldots, b_m). 
		\end{eqnarray*}
		 For a representation $\rho$ derived from a YBO $R$, let $B_i$ and $B_i'$ denote $\rho(b_i)$ and $\rho(b_i')$ as before. 
				We  
		show that each vector $v\in E(\vec{B}, \lambda)$ is a $\lambda$ eigenvector for each $B_k'$, 
		and similarly, 
		that each vector in $v\in E(\vec{B}', \lambda)$ is a $\lambda$ eigenvector for each $B_k$. % $b_k$. 
		Let $v \in E(\vec{B}, \lambda)$. By definition, $v$ is a $\lambda$ eigenvector of $B_k$ for each $k=1, \ldots, m$. It follows directly by definition of $\vec B'$ that $v$ is an eigenvector of $B'_k$ for $k=1,\ldots, i-1, i+1, \ldots , m$. For $k=i$, we have that 
		\begin{eqnarray*}
				B'_iv =  B_i B_{i+1} B_i^{-1}v = \frac{1}{\lambda}B_i B_{i+1}v = B_iv = \lambda v. 
		\end{eqnarray*}
		This shows the inclusion $E(\vec{B}, \lambda) \subset E(\vec{B}', \lambda)$. To show the reverse inclusion, consider $v\in E(\vec{B}', \lambda)$. Then $v$ is a $\lambda$ eigenvector of $B'_k$ for all $k$. It follows that $v$ is a $\lambda$ eigenvector of $B_k$ for $k=1,\ldots ,i, i+2,\ldots, m$. For $k=i+1$, since $B_iB_{i+1}B^{-1}_iv = \lambda v$, we have that $B^{-1}_iv$ is a $\lambda$ eigenvector of $B_{i+1}$. Since $B_i = B'_{i+1}$ by assumption , it follows that $v$ is a $\lambda$ eigenvector of $B_i$, and therefore a $\lambda^{-1}$ eigenvector of $B^{-1}_i$. Therefore, $\frac{1}{\lambda}B_{i+1} v = v$, from which it follows that $E(\vec{B}', \lambda) \subset E(\vec{B}, \lambda)$. 
%		To prove invariance under \eqref{eqn:slide2}, notice that 
%		\begin{eqnarray*}
%				(b_1, \ldots, b_n) &\sim&  (b_1 , \ldots, b_{i-1}, b_{i+1}, b^{-1}_{i+1} b_i b_{i+1} , b_{i+2} \ldots, b_n), 
%		\end{eqnarray*}
%		can be obtained  by applying \eqref{eqn:slide1} to $(b_1 , \ldots, b_{i-1}, b_i b_{i+1} b_i^{-1} , b_i, b_{i+2}, \ldots, b_n)$, it follows that \eqref{eqn:slide2} is a consequence of  \eqref{eqn:slide1}, which is in line with Remark~\ref{rmk:slide_equivalence}.
\end{proof}

\section{Quandle colorings as eigenvectors}\label{sec:quandle_col}

First, recall a \textit{quandle}~\cite{Joyce82,Matveev82}, is a non-empty set $X$ with a binary operation $*:X\times X\to X$ satisfying the following axioms.
\begin{itemize}
\item[(1)] For any $a\in X$, we have $a*a=a$.
\item[(2)] For any $a\in X$, the map $S_a:X\to X$ defined by $S_a(x)=x*a$ is a bijection.
\item[(3)] For any $a,b,c\in X$, we have $(a*b)*c=(a*c)*(b*c)$.
\end{itemize}
For details of examples, quandle colorings of knots, and quandle cocycle invariants, we refer to \cite{CJKLS}.

In this section we relate the number of colorings of surface braids by quandles and the eigenspace invariant defined in the preceding section of the linearlized quandles (quandle rings). 
  A {\it quandle ring} $\mathbb k [X]$ for a ring $\mathbb k$ and a quandle $X$ is a non-associative ring generated by elements of $X$ with multiplication obtained by  linearly extending the quandle operation on $X$.  

Let $X$ be a finite quandle.
Define \cite{Brieskorn} the action of the $n$-string  braid group $\mathbb B_n$ on $X^n$ by 
\begin{eqnarray*}
\sigma_i  (x_1\ldots, x_n)  &=& 
 (x_1,  \ldots, x_{i-1},  x_{i+1}, x_i * x_{i+1},  x_{i+2}, \ldots, x_n) , \\
\sigma_i^{-1} ( x_1,  \ldots, x_n)  &=& 
(x_1,  \ldots, x_{i-1}, x_{i+1} * x_{i}, x_i, x_{i+2},  \ldots, x_n) .
\end{eqnarray*}

\begin{lemma}\label{lem:fundQ}
Let $S$ be a surface braid with the set of black vertices $\Sigma=(y_1, \ldots, y_m)$.
Let $\vec{b}=(b_1, \ldots, b_m)$ be its braid system.
Then the fundamental quandle of $S$ is generated by $x_1, \ldots, x_n$ with relations
$b_i (x_1, \ldots, x_n)=(x_1, \ldots, x_n)$, $i=1, \ldots, m$.
\end{lemma}

\begin{proof}
Recall the notation
$$\vec{b}=(b_1, \ldots, b_m)=(w_1^{-1} \sigma^{\epsilon_1}_{k_1} w_1, \ldots, w_m^{-1} \sigma^{\epsilon_m}_{k_m} w_m ). $$
In \cite{CJKLS}, using \cite{Kamada:SB}, the presentation of the fundamental quandle is given by generators 
$x_1, \ldots, x_n$ and relations 
$[ w_i (x_1, \ldots, x_n) ]_{k_i} = [ w_{i} ( x_1, \ldots, x_n)]_{k_{i}+1}$, 
$i=1, \ldots, n$, where the subscript $k_i$ indicates $k_i$-th term of an $n$-tuple. 
Set $(x_1', \ldots, x_n')=w_i (x_1, \ldots, x_n)$. Then the relations are written as 
$x_{k_i}' = x_{k_i +1}'$. 
A small circle around $y_i$ intersects an arc of the chart of $S$ labeled $x_{k_i}$,
and the corresponding monodromy is 
$$
\sigma_{k_i} (x_1', \ldots, x_n') = (x_1', \ldots, x_{k_i-1}', x_{k_{i+1}}' , x_{k_i}' * x_{k_{i+1}}', x_{k_i+2}', \ldots, x_n') ,
$$ 
where the two sheets around a branch point receive the same label, resulting in 
the relation $ x_{k_i}'=x_{k_{i} +1}'$. 
Hence this relation corresponds to $\sigma_{k_i} (x_1', \ldots, x_n') =(x_1', \ldots, x_n') $.
Thus the relation 
$[ w_i (x_1, \ldots, x_n) ]_{k_i} = [ w_{i} ( x_1, \ldots, x_n)]_{k_{i}+1}$ 
 can be rewritten as 
$(w_i^{-1} \sigma_{k_i} w_i )  (x_1, \ldots, x_n)= (x_1, \ldots, x_n)$, that is,
$b_i  (x_1, \ldots, x_n)= (x_1, \ldots, x_n)$, as desired.
\end{proof}

This lemma implies the following, which was used in a different form in \cite{CJKLS}. 

\begin{corollary}\label{cor:color}
A quandle vector $(x_1 , \ldots, x_n)$  of a quandle $X$ at the base point gives rise to a quandle coloring of a surface braid $S$
if and only if $b_i (x_1 , \ldots, x_n)=(x_1 , \ldots, x_n)$ for all $i=1, \ldots, m$.
\end{corollary}

Let $V= \mathbb \C [X]$ be the quandle ring of a quandle $X$. 
The linealization of the set-theoretic YBO $(x,y) \mapsto (y, x*y)$ for $X$, 
$R : V \otimes V \rightarrow  V \otimes V$, is obtained by linearly extending 
$x \otimes y \mapsto y \otimes x*y$ defined on basis vectors $x, y \in X$. 
This is a linearization of $(x,y) \mapsto (y, x*y)$ defined on $X \times X$ to 
$\mathbb C [X \times X] \cong \mathbb C [X] \otimes  \mathbb C [X] $.

The  braid group $\mathbb B_n$ acts on $X^n$ as permutations defined by 
$$\sigma_i (x_1, \ldots, x_n) = (x_1, \ldots, x_{i-1}, x_{i+1}, x_i * x_{i+1}, x_{i+2}, \ldots, x_n), $$
inducing a homomorphism $\mathbb B_n \rightarrow \Sigma_k$,  where $k=|X| n$. 
The YBO $R : V \otimes V \rightarrow  V \otimes V$ defined above gives  a representation 
$\rho: \mathbb B_n \rightarrow {\rm Aut} ( V^{\otimes n}) $
by the map $\sigma_j \mapsto \mathbb 1^{\otimes (j-1)} \otimes R \otimes \mathbb 1^{\otimes (n-j-1)}$.
Then $\rho$ is regarded as 
the linearlization of the action of $\mathbb B_n$ on $X^n $ by 
the isomorphisms $\C [X^n]=\C [X]^{\otimes n} = V^{\otimes n}$. 

\begin{proposition}
We have that the number of quandle colorings ${\rm col}_X (S)$ of a surface braid $S$ by a quandle $X$ is less than or equal to the rank of $E(\vec{B}, 1)$, ${\rm col}_X (S) \leq {\rm rank}E(\vec{B}, 1)$.
\end{proposition}

\begin{sloppypar}
\begin{proof}
By Corollary~\ref{cor:color}, a coloring is uniquely determined by 
the simultaneous equations $b_i(x_1, \ldots, x_n)=(x_1, \ldots, x_n)$.
The solutions $(x_1, \ldots, x_n)$ form sub-basis vectors of $E(\vec{B}, 1)$, hence the result follows.
\end{proof}
\end{sloppypar}

This proposition shows that the eigenspace invariants defined in this paper are in general nontrivial.

Below  we use the following well known lemma on the permutation representation induced by a $G$-set $Y$, where $G$ is a finite group and $Y$ is a finite set.
 In this case the permutation representation $\C [Y]$ is obtained from the $G$-action on $Y$ by linearly extending the action on $Y$ to $\C [Y]$.
Let $\C [Y]^G$ denote the space of  $G$-invariant vectors, $\C [Y]^G = \{ v \in \C [Y] \ | \ g  \cdot v =v \ {\rm for \ all} \  g \in G \}$. Note that an invariant subspace $W$ is defined by $G \cdot W \subset W$ and the space of $G$-invariant vectors is a special example of invariant subspace. 

\begin{lemma}\label{lem:GrepNote}
Let $O_1, \ldots, O_t$ be $G$-orbits in $Y$. Set $v_i=\sum_{y \in O_i} y \in \C[Y]$, $i=1, \ldots, t$. 
Then $\{ v_i \}_{i=1}^t$ span $\C [Y]^G$.
In particular,  ${\rm dim} (\C [Y]^G ) $ is equal to the number of $G$-orbits of $Y$.
\end{lemma}

We examine $E(\vec{B}, 1)$ in more detail for the quandle ring $V=\C [X]$ for a quandle $X$ and a YBO $R$ defined from the quandle structure as above.
Let $\rho: \mathbb B_n \rightarrow {\rm Aut} ( V^{\otimes n}) $ be the representation defined by $R$. 
Recall that  $\rho(b)$ is a linearlization of the map  $f_b : X^n \rightarrow X^n$ for $b \in \mathbb B_n$
defined by composing maps of type
$$\sigma_i (x_1, \ldots, x_n) = (x_1, \ldots, x_{i-1}, x_{i+1}, x_i * x_{i+1}, x_{i+2}, \ldots, x_n),$$
and
$$\sigma^{-1}_i (x_1, \ldots, x_n) = (x_1, \ldots, x_{i-1}, x_{i+1}*x_1, x_i, x_{i+2}, \ldots, x_n).$$

Let $\vec{b}=( b_1, \ldots, b_m)$ be a braid system of  a surface braid $S$.
The braid $b_i$  induces the map $f_i$ on $X^n$ as before.
Let $H=\langle f_1, \ldots, f_m \rangle \leq \Sigma_k$, 
which acts on $\C[X^n] \cong V^{\otimes n}$ through the action on $X^n$. 
 Let $(\C [X^n] )^H$ be the space of invariant vectors of $H$.
Let $E(\vec B, 1)= \cap_{i=1}^m E(\rho (b_i) , 1)$ as before.

Identify elements $(x_1, \ldots, x_n ) \in X^n$ with a simple tensor vector $x_1 \otimes \cdots \otimes x_n \in  V^{\otimes n}$. Let $O_1, \ldots, O_t$ be orbits of $H$-action on $X^n$, and for each $i$, 
let $v_i=\sum_{x \in O_i} x \in  \C[X^n] $ regarded as elements of $V^{\otimes n}$, which we call an {\it orbit vector} of $H$. 

\begin{lemma}\label{lem:color}
We have $E(\vec B, 1) = \C [X^n]^H$. 
In particular, $E(\vec B, 1) $ is spanned by orbit vectors $\{ v_i\}_{i=1}^t$.
Moreover, the set of colorings of $S$ by $X$ is in bijection with the set of simple tensor orbit vectors. 
\end{lemma}
\begin{proof}
Set $Y=X^n$ in Lemma~\ref{lem:GrepNote}, 
on which $\mathbb B_n$ acts.
If  $v \in E(\vec B , 1)$, then $B_i v=v$ for any $B_i=\rho (b_i)$, $i=1, \ldots, m$, hence $v \in  \C [X^n]^H$.
Conversely, if $v \in \C [X^n]^H$, then $B_iv=v$ since $b_i \in H$, hence $v \in E(\vec B , 1)$.
Therefore the statement follows from  Lemma~\ref{lem:GrepNote}.

The tuple $x=(x_1, \ldots, x_n)\in X^n$ gives a coloring of $S$ by $X$ if and only if $b_i(x)=x$, which is equivalent to the condition that $B_i x=x$ under the correspondence $(x_1, \ldots, x_n)\mapsto x_1 \otimes  \cdots \otimes  x_n$.
\end{proof}

%\begin{corollary}
%Let $b\in \mathbb B_n$, $B=\rho(b)$, and $E(B, 1)$ be the $1$-eigenspace as above.
%Then the number of colorings of the closed braid $\hat b$  by $X$ coincides with ${\rm dim} E(B,1)$ if and only if $B=\mathbb 1$,
%the identity matrix. In this case, ${\rm dim}E(B,1)=k=|X|^n$.
%\end{corollary}
%
%\begin{proof}
%Let $m=1$ in Lemma~\ref{lem:color}. 
%If  $B=\mathbb 1$, then ${\rm dim}E(B,1)=k$, and every simple tensor  gives a coloring, so that the number of colorings is also $k$.
%Suppose $B\neq \mathbb 1$. Then there is a simple tensor that is not fixed.
%then its orbit contains nonsimple tensor vector, that does not correspond to a coloring, hence 
%the number of colorings is less than ${\rm dim} E(B,1)$. 
%\end{proof}

\begin{example}\label{ex:22}
{\rm

%It is known \cite{Kamada:SB} 
% that a braid system of a spun $(2,\ell)$-torus knot, for a positive integer $\ell$,  is given by 
%\begin{equation}
%\vec{b}=(b_1, b_2, b_3, b_4)=( \sigma_2, \sigma_2^{-1}, \sigma_1^{-\ell} \sigma_2  \sigma_1^{\ell}, \sigma_1^{-\ell} \sigma_2^{-1}  \sigma_1^{\ell} ). \label{eqn:tknot}
%\end{equation}
%Let $H=\langle b_i \rangle_{i=1}^4$ as above.

%For a spun trefoil, for example, the ribbon chart consists of 
%two free edges of label 2, one outside and the other inside of  three circles of label 1.
%(Section 29.4 in \cite{Kamada:SB}, Figure 29.3.)
%Its braid system is given by
%$$\vec{b}=(b_1, b_2, b_3, b_4)=( \sigma_2, \sigma_2^{-1}, \sigma_1^{-3} \sigma_2  \sigma_1^{3}, \sigma_1^{-3} \sigma_2^{-1}  \sigma_1^{3} ). $$
 We consider the surface braid $S$ of braid index $3$ and the braid system
 $\vec b=(b_1, b_2)=(\sigma_2, \sigma_2^{-1}) $, where $\sigma_2 \in \mathbb B_3$, 
 so that $H=\langle \sigma_2 \rangle$.
 The chart of $S$ is represented by a single linear edge labeled by $2$ with two end point black vertices 
 corresponding to $\sigma_2$. It is a stabilization of the trivial surface of braid index 2 with empty chart, see the last section.

Let $X={\sf R}_3$ be the dihedral quandle of order 3, $x*y=2y-x$ in $\Z_3$. 
Then $V=\C [X]$ is spanned by $e_{i}$, $i \in \Z_3=\{0,1,2\}$,
and $V^{\otimes 3}=\C [X^3]$, spanned by $e_{i,j,k}:=e_i \otimes e_j \otimes e_k$, $(i,j,k)\in \Z_3$.
%We consider $\vec b=(b_1, b_2)=(\sigma_2, \sigma_2^{-1}) $, where $\sigma_2 \in \mathbb B_3$,
%so that $H=\langle \sigma_2 \rangle$.
%

Since $\sigma_2 (x,y,z)=(x,z, 2z-y)$ on $X^3$, we have that $\C [X^3]^H$ is spanned by
$$v_1(i,j)=e_i \otimes e_j \otimes e_j, \quad v_2(i)=\sum_{k \in \Z_3} e_i \otimes e_k \otimes e_{k+1},\quad
v_3(i)=\sum_{k \in \Z_3} e_i \otimes e_k \otimes e_{k-1}$$
for $i,j \in \Z_3$.
Hence for $\rho(b)=B$ we have $${\rm dim}E(B,1)=9 + 3 + 3=15.$$
The number of colorings of the surface braid $S$ represented by $\vec b$  is the number of simple tensors $v_1(i,j)$, which is $9$.
Thus in this case we have ${\rm col}_X(S) < {\rm dim}E(\vec B, 1)$. 

 The closure of $S$ (see the last section) is a trivial knotted sphere and the colorings above are trivial. This example shows that the eigenspace invariant is different from the number of colorings by quandles.

}
\end{example}

\section{Examples from the bracket YBO}\label{sec:bracket}

In this section we examine the YB eigenvalue invariant for the YBO used for the Kauffman bracket polynomial, that yields the Jones polynomial. In particular, it is shown that the invariant defined from this YBO distinguishes surface braids
whose closures are spun $(2, \ell)$-torus knots.

%The matrix that defines the Jones polynomial is given by
%	\begin{eqnarray*}
%				J =  \begin{bmatrix}
%					t^{\frac{1}{2}} & 0 & 0 & 0\\
%					0 & 0 & t & 0\\
%					0 & t & t^{\frac{1}{2}}-t^{\frac{3}{2}} & 0\\
%					0 & 0 & 0 & t^{\frac{1}{2}}
%				\end{bmatrix}
%\end{eqnarray*}
%%% Is it $-t$ instead of $t $? 

\begin{figure}[htb]
\begin{center}
\includegraphics[width=3in]{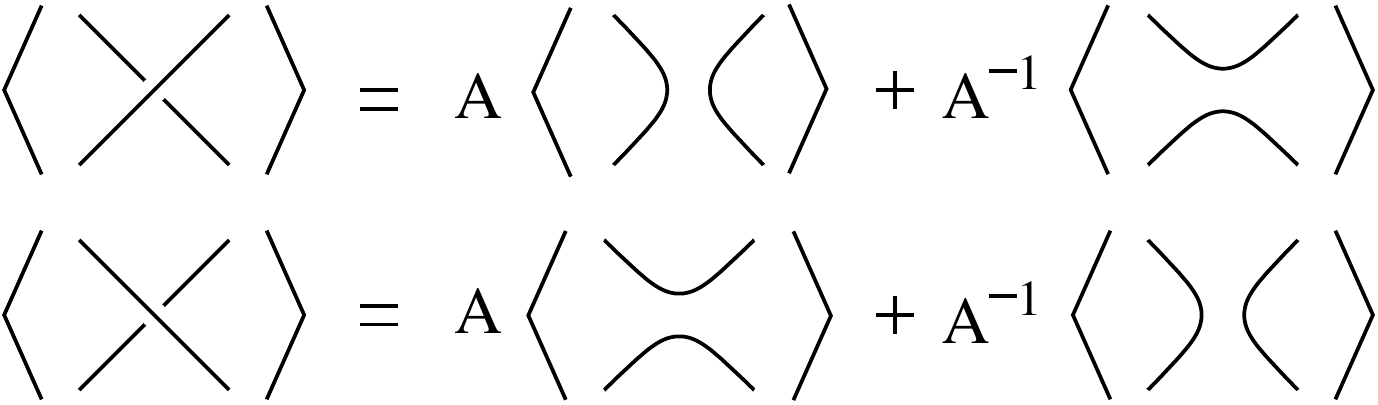}
\end{center}
\caption{Skein relation of Kauffman bracket}
\label{skeinbracket}
\end{figure}

The Kauffman bracket polynomial \cite{K&P}
for knots up to regular isotopy is defined by the skein relation (see below, and depicted in Figure~\ref{skeinbracket}).
%{  The bracket value for the unknot diagram with no crossing is set to be $1$. }
In this section we assume that $\mathbb k$ contains  $\sqrt{-1}$  and an invertible element $A$
that serves as a variable in the Kauffman bracket. 

Let $\cup$ and $\cap$ be pairing used in the bracket skein relation defined in \cite{K&P}
for basis elements $e_1$, $e_2$ of $V={\mathbb k} \oplus {\mathbb k}$ by $\cup(e_i\otimes e_i)=0$ for $i=1,2$ and 
$$\cup(e_1 \otimes e_2)=\sqrt{-1} A , \quad \cup(e_2 \otimes e_1)=- \sqrt{-1} A^{-1} . $$
The copairing is defined by 
$$\cap (1)=\sqrt{-1} A e_1 \otimes e_2 - \sqrt{-1} A^{-1} e_2 \otimes e_1. $$
The Kauffman bracket YBO  is defined by the bracket skein relation 
$ R= A \mathbb 1 + A^{-1} \cap \cup$.

\begin{lemma}
The characteristic polynomial 
of $ R= A \mathbb 1 + A^{-1} \cap \cup$ is $(\lambda -A)^3(\lambda + A^{-3} )$,
so that its eigenvalues are $\lambda=A$ with multiplicity 3, and a single root $- A^{-3}$.
\end{lemma}

\begin{proof}
We order the basis of $V \otimes V$ by 
$e_1 \otimes e_1$, $e_1 \otimes e_2$, $e_2 \otimes e_1$, $e_2 \otimes e_2$ in this order for vector and matrix notations below.
The YBO of the bracket, according to the skein definition, is written  as 
	\begin{eqnarray*}
		R = \begin{bmatrix}
		A & 0 & 0 & 0\\
		0 & 0 & A^{-1} & 0\\
		0 & A^{-1} & A-A^{-3} & 0\\
		0 & 0 & 0 & A
		\end{bmatrix} .
\end{eqnarray*}
The determinant of $\lambda I - R$ for the identity matrix $I$ is computed as 
$$ (\lambda - A )^2 [ \lambda (\lambda - (A-A^{-3}) ) - A^{-2}] $$
and the result follows.
\end{proof}

%	I think the following goes through ``normally'' since the ring is a PID. At least using the results on Lang's algebra book it seems that things should go through. 
\begin{lemma}\label{lem:bracketeigen}
Let $ R= A \mathbb 1 + A^{-1} \cap \cup$ be as above.
%Then the eigenspace for the eigenvalue $\lambda = A$ is spanned by 
%$$
%\begin{array}{lllll}
%u_1 & := & (1, 0, 0, 0)^T &=& e_1 \otimes e_1, \\
%u_2 & := & (0,0,0,1)^T &=& e_2 \otimes e_2, \\
%u_3 &:=& (0, 1, A^2, 0)^T &=&  e_1 \otimes e_2 + A^2 e_2 \otimes e_1,
%\end{array}
%$$
%where ${}^T$ denotes the transpose.
We have an eigendecomposition $ \langle u_1,u_2,u_3\rangle \oplus \langle u_4\rangle $ 
that is a rank $4$ free submodule of  $\oplus_4\mathbb C[A,A^{-1}]$, where 
$$
\begin{array}{lllll}
u_1 & := & (1, 0, 0, 0)^T &=& e_1 \otimes e_1, \\
u_2 & := & (0,0,0,1)^T &=& e_2 \otimes e_2, \\
u_3 &:=& (0, 1, A^2, 0)^T &=&  e_1 \otimes e_2 + A^2 e_2 \otimes e_1, 
\end{array}
$$
spans the eigenspace %are eigenvectors
 for the eigenvalue $\lambda = A$, and 
$$u_4 = (0,-A^2,1,0)^T = -A^2e_1\otimes e_2 + e_2\otimes e_1,$$
spans the 
eigenspace for the eigenvalue $\lambda = -A^{-3}$. 
\end{lemma}

\begin{proof}
We compute 
\begin{eqnarray*}
	AI-	R = \begin{bmatrix}
		0 & 0 & 0 & 0\\
		0 & A & -A^{-1} & 0\\
		0 & -A^{-1} & A^{-3} & 0\\
		0 & 0 & 0 & 0
		\end{bmatrix} .
\end{eqnarray*}
Let ${\bf v}=(x_1, x_2, x_3, x_4)^T$, then $(AI-R){\bf v}=0$ gives
$$Ax_2 - A^{-1} x_3=0  \quad {\rm  and} \quad - A^{-1} x_2 + A^{-3} x_3=0, $$
 which yield the solution space spanned by vectors as stated. A direct inspection shows that the vectors are linearly independent. 
For the eigenvalue $-A^{-3}$, we obtain
$$ -A(1+A^{-4})x_1=0, \quad  - A^{-3} x_2 - A^{-1} x_3 =0, \quad - A^{-1} x_2 - A x_3 =0,  \quad -A(1+A^{-4})x_4=0 .$$
A similar computation shows that the given $u_4$ is an eigenvector associated to $\lambda = -A^{-3}$. Since $\mathbb C[A,A^{-1}]$ is a PID, it follows that submodules of $\oplus_4\mathbb C[A,A^{-1}]$ are free as well, which shows that $\langle u_1,u_2,u_3\rangle$  forms a rank 3 free submodule,  and $\langle u_4\rangle$ forms a rank 1 free submodule, that are linearly independent.
\end{proof}

\begin{lemma}\label{lem:induction}
Let $R$ be the bracket YBO as above. 
For all positive integers $n$, we have
$$ R^n=A^n \mathbb 1 + \frac{\sum_{i=1}^n(-1)^{i+1}A^{4(n-i)}}{A^{1+3(n-1)}}\cap \cup. $$
\end{lemma}
\begin{proof}
		We proceed by induction on $n$. For $n=1$ one sees that the formula gives $R = A\mathbb 1 + A^{-1}\cap \cup$ as required. Suppose now that the formula holds true for some $n$.  Recall that 
		 $\cup\cap = -A^2-A^{-2}$, which can be checked directly using $\cup$ and $\cap$ given above. Then, we have
		\begin{eqnarray*}
				R^{n+1} &=& (A\mathbb 1 + A^{-1}\cap\cup)(A^n \mathbb 1 + \frac{\sum_{i=1}^n(-1)^{i+1}A^{4(n-i)}}{A^{1+3(n-1)}}(\cap \cup ))\\
				&=& A^{n+1}\mathbb 1 + A\frac{\sum_{i=1}^n(-1)^{i+1}A^{4(n-i)}}{A^{1+3(n-1)}}\cap \cup  + A^{n-1}\cap\cup + A^{-1}\frac{\sum_{i=1}^n(-1)^{i+1}A^{4(n-i)}}{A^{1+3(n-1)}}
				\cap\cup \cap\cup \\
				&=& A^{n+1}\mathbb 1 + [A\frac{\sum_{i=1}^n(-1)^{i+1}A^{4(n-i)}}{A^{1+3(n-1)}} + A^{n-1} + A^{-1}\frac{\sum_{i=1}^n(-1)^{i+1}A^{4(n-i)}}{A^{1+3(n-1)}}(-A^2-A^{-2})] \cap\cup. 
		\end{eqnarray*}
		Therefore, the inductive step is completed once we show the equality 
		\begin{eqnarray}\label{eqn:inductive_step}
		\begin{aligned}
			\lefteqn{\frac{\sum_{i=1}^{n+1}(-1)^{i+1}A^{4(n+1-i)}}{A^{1+3n}}}\\	&=& A\frac{\sum_{i=1}^n(-1)^{i+1}A^{4(n-i)}}{A^{1+3(n-1)}} + A^{n-1} + A^{-1}\frac{\sum_{i=1}^n(-1)^{i+1}A^{4(n-i)}}{A^{1+3(n-1)}}(-A^2-A^{-2}).
		\end{aligned}
		\end{eqnarray}
		This is seen to hold by direct computation. In fact, % for 
		the 	
		right-hand-side is equal to %we have
%	\begin{eqnarray*}
%				\lefteqn{A\frac{\sum_{i=1}^n(-1)^{i+1}A^{4(n-i)}}{A^{1+3(n-1)}} + A^{n-1} + A^{-1}\frac{\sum_{i=1}^n(-1)^{i+1}A^{4(n-i)}}{A^{1+3(n-1)}}(-A^2-A^{-2})}\\
%				&=& 
\begin{eqnarray*}
{ } & & A\frac{\sum_{i=1}^n(-1)^{i+1}A^{4(n-i)}}{A^{1+3(n-1)}} + A^{n-1} - A\frac{\sum_{i=1}^n(-1)^{i+1}A^{4(n-i)}}{A^{1+3(n-1)}} - \frac{\sum_{i=1}^n(-1)^{i+1}A^{4(n-i)}}{A^{1+3(n-1)+3}  }  \\
				&=& A^{n-1}-\frac{\sum_{i=1}^n(-1)^{i+1}A^{4(n-i)}}{A^{1+3n}}, 
		\end{eqnarray*}
		and we can see by direct inspection that 
		\begin{eqnarray*}
		A^{n-1}-\frac{\sum_{i=1}^n(-1)^{i+1}A^{4(n-i)}}{A^{1+3n}} &=& \frac{\sum_{i=1}^{n+1}(-1)^{i+1}A^{4(n+1-i)}}{A^{1+3n}}
		\end{eqnarray*} 
		therefore showing that \eqref{eqn:inductive_step} holds, and consequently completing the proof of the lemma. 
\end{proof}

For concrete computations of examples below, we provide the following lemmas.

\begin{lemma}\label{lem:VEA}
The images 
$(\cap\cup \otimes \mathbb 1) (V \otimes E_A) =(\cap\cup \otimes \mathbb 1) \langle e_i \otimes u_j \rangle_{i=1,2, j=1,2,3}$ are computed as 
\begin{eqnarray*}
 e_1 \otimes u_1 &=& e_1 \otimes e_1 \otimes e_1
 \stackrel{ (\cap\cup \otimes \mathbb 1) }{\longrightarrow}
 0,\\
e_1 \otimes u_2 &=& e_1 \otimes e_2 \otimes e_2 
 \stackrel{ (\cap\cup \otimes \mathbb 1) }{\longrightarrow}
  u_{12}   := -A^2 e_1 \otimes e_2 \otimes e_2 + e_2 \otimes e_1 \otimes e_2   , \\
e_1 \otimes u_3 &=&  e_1 \otimes ( e_1 \otimes e_2 + A^2 e_2 \otimes e_1) 
 \stackrel{ (\cap\cup \otimes \mathbb 1) }{\longrightarrow}
 u_{13}   :=  -A^4 e_1 \otimes e_2 \otimes e_1 + A^2  e_2 \otimes e_1 \otimes e_1 ,  \\
  e_2 \otimes u_1 &=& e_2 \otimes  e_1 \otimes e_1
 \stackrel{ (\cap\cup \otimes \mathbb 1) }{\longrightarrow}
 e_1 \otimes e_2 \otimes e_1 - A^{-2} e_2 \otimes e_1 \otimes e_1 = - A^{-4} u_{13}, \\
 e_2 \otimes u_2 &=&e_2 \otimes e_2 \otimes e_2
 \stackrel{ (\cap\cup \otimes \mathbb 1) }{\longrightarrow}
 0,  \\
 e_2 \otimes u_3 &=& e_2 \otimes ( e_1 \otimes e_2 + A^2 e_2 \otimes e_1) 
 \stackrel{ (\cap\cup \otimes \mathbb 1) }{\longrightarrow}
e_1 \otimes e_2 \otimes e_2 - A^{-2} e_2 \otimes e_1 \otimes e_2 = -A^{-2} u_{12}.
\end{eqnarray*}
In particular, the submodule $(\cap\cup \otimes \mathbb 1) (V \otimes E_A) $ is spanned by 
$u_{12}$ and $u_{13}$. 
%\begin{eqnarray*}
%& & u_{12} := -A^2 e_1 \otimes e_2 \otimes e_2 + e_2 \otimes e_1 \otimes e_2, \\
%& & u_{13} :=  -A^4 e_1 \otimes e_2 \otimes e_1 + A^2  e_2 \otimes e_1 \otimes e_1 .
%%& & u_{23} := e_1 \otimes e_2 \otimes e_2 - A^{-2} e_2 \otimes e_1 \otimes e_2 .
%\end{eqnarray*}
\end{lemma}

\begin{proof}
One computes that 
\begin{eqnarray*}
(\cap\cup) (e_1 \otimes e_2) &=& -A^2 e_1 \otimes e_2 + e_2 \otimes e_1, \\
(\cap\cup) (e_2 \otimes e_1) &=&  e_1 \otimes e_2 - A^{-2}  e_2 \otimes e_1, 
\end{eqnarray*}
and other basis terms  vanish. 
%Note that $V \otimes E_{A} = \langle  e_i \otimes u_j \rangle_{i=1,2, \ j=1,2,3}$.
Using 
 Lemma~\ref{lem:bracketeigen} we obtain the images as stated.
\end{proof}

Similar computations imply the following for $\lambda=-A^{-3}$.

\begin{lemma}\label{lem:VEA-3}
The images 
$(\cap\cup \otimes \mathbb 1) (V \otimes E_{A^{-3} } )=(\cap\cup \otimes \mathbb 1) \langle e_i \otimes u_4 \rangle_{i=1,2}$ are computed as 
\begin{eqnarray*}
 e_1 \otimes u_4 &=& -A^2 e_1 \otimes e_1 \otimes e_2 + e_1 \otimes e_2 \otimes e_1 
 \stackrel{ (\cap\cup \otimes \mathbb 1) }{\longrightarrow}
 -A^2 e_1 \otimes e_2 \otimes e_1  + e_2 \otimes e_1 \otimes e_1 ,\\
 e_2 \otimes u_4 &=& -A^2 e_2 \otimes e_1 \otimes e_2 + e_2 \otimes e_2 \otimes e_1 
 \stackrel{ (\cap\cup \otimes \mathbb 1) }{\longrightarrow}
 -A^2 e_1 \otimes e_2 \otimes e_2  + e_2 \otimes e_1 \otimes e_2 .
\end{eqnarray*}
\end{lemma}

\begin{figure}[htb]
\begin{center}
\includegraphics[width=1.8in]{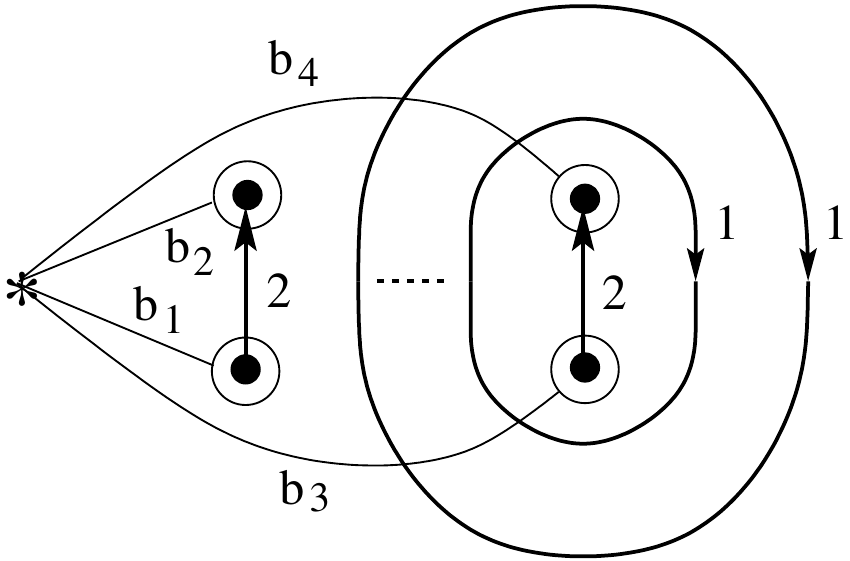}
\end{center}
\caption{Charts of spun $T(2, \ell)$ torus knots}
\label{spunT}
\end{figure}

Recall that a braid system of a spun $(2, - \ell)$-torus knot, for a positive integer $\ell$,  is given by 
\begin{equation}
\vec{b}=(b_1, b_2, b_3, b_4)=( \sigma_2, \sigma_2^{-1}, \sigma_1^{\ell} \sigma_2  \sigma_1^{- \ell}, \sigma_1^{\ell} \sigma_2^{-1}  \sigma_1^{-\ell} ). \label{eqn:tknot}
\end{equation}
 The negative sign in $(2, -\ell)$ is  for our convenience for computations below. 
 For $(2,\ell)$, the signs of exponents involving $\ell$ are all reversed for the braid system.

For a  surface braid whose closure is a  
spun trefoil, for example,  a  chart consists of 
two free edges of label 2, one outside and the other inside of  three circles of label 1.
(Section 29.4 in \cite{Kamada:SB}, Figure 29.3.)
Its braid system is given by
$$\vec{b}=(b_1, b_2, b_3, b_4)=( \sigma_2, \sigma_2^{-1}, \sigma_1^{3} \sigma_2  \sigma_1^{-3}, \sigma_1^{3} \sigma_2^{-1}  \sigma_1^{-3} ). $$
The signs of $\ell$ are again reversed comparing to  \cite{Kamada:SB}.
This is reflected by reversed orientation of nested circles labeled by 1.

\begin{proposition}\label{prop:twistspun}
For the braid system \eqref{eqn:tknot} 
 of the surface braid whose closure 
 represents a spun $(2,-\ell)$-torus knot, the submodule
			$E_{(\vec{B}_\ell, A)}$ has rank $4$ for each $\ell$, and they are therefore mutually isomorphic. 
			However, for  
			 $\ell_1 \neq \ell_2$ we have that $E_{(\vec{B}_{\ell_1}, A)} \neq E_{(\vec{B}_{\ell_2}, A)}$. The module $E_{(\vec{B}_\ell, -A^{-3})}$ is trivial. 
\end{proposition}

\begin{proof}
We determine $E_{(\vec{B}, \lambda)}$ for $\lambda=A$ from Lemma~\ref{lem:bracketeigen}.
We use basis eigenvactors $u_1:=(1, 0, 0, 0)^T$, $u_2:=(0,0,0,1)^T$ and $u_3:=(0, 1, A^2, 0)^T$ 
for $\lambda=A$ as in Lemma~\ref{lem:bracketeigen}.

For $B_1=\rho(b_1)= \rho(\sigma_2)=\mathbb 1 \otimes R$, from Lemma~\ref{lem:bracketeigen}
we have 
$$E(B_1, A) \supset E_{(B_1, \lambda)}=V \otimes E_{A} = \langle 
e_i \otimes u_j \rangle_{i=1,2, \ j=1,2,3} .
%e_2 \otimes u_1, e_2 \otimes u_2, e_2 \otimes u_3  \rangle . 
$$
For
$B_2=\rho(b_2)=\rho(\sigma_2^{-1})=\mathbb 1 \otimes R^{-1}$
we obtain $ E_{(B_2, A) } = E_{ (B_1, A) } $  
since the eigenspace of the inverse matrix $R^{-1}$ for $\lambda^{-1}$ is the same as the eigenspace of $R$ for $\lambda$. 

Next we compute $E_{(B_3, A)}$. By Lemma~\ref{lem:conjeigen} we have
$$E_{(B_3, A)}= W_3 (E_{(2,A)})=(R \otimes \mathbb 1)^\ell (V \otimes E_A)=(R^\ell \otimes \mathbb 1)(V \otimes E_A)  . $$
Let  
$p_\ell = \frac{\sum_{i=1}^\ell(-1)^{i+1}A^{4(\ell -i)}}{A^{1+3(\ell -1)}}$. 
By Lemmas~ \ref{lem:induction},  %and  \ref{lem:VEA},  
we compute
\begin{eqnarray*}
E_{(B_3, A)}= E_{(B_4, A)}=  (R^\ell \otimes \mathbb 1) (V \otimes E_A) 
&=& [  \ A^\ell \mathbb 1^{\otimes 3}  + p_\ell   (\cap \cup \otimes \mathbb 1) \  ]  (V \otimes E_A) \\
&=& [ \ A^\ell  \mathbb 1^{\otimes 3}+ p_\ell   (\cap \cup \otimes \mathbb 1) \  ]\langle e_i \otimes u_j \rangle_{i=1,2, \ j=1,2,3}  .
\end{eqnarray*}
The image of each term is computed using  Lemma~\ref{lem:VEA} as follows.
\begin{eqnarray*}
e_1 \otimes u_1 & \mapsto & A^\ell\  e_1 \otimes e_1 \otimes e_1 , \\
e_1 \otimes u_2 & \mapsto & A^\ell \ e_1 \otimes e_2 \otimes e_2 + p_\ell \ u_{12} \\
& & = (A^\ell - A^2  p_\ell ) e_1 \otimes e_2 \otimes e_2 +  p_\ell  \ e_2 \otimes e_1 \otimes e_2 , \\
e_1 \otimes u_3 & \mapsto & A^\ell (  e_1 \otimes e_1 \otimes e_2 + A^2 \  e_1 \otimes e_2 \otimes e_1) 
+ p_\ell\  u_{13} \\
& & = A^\ell \   e_1 \otimes e_1 \otimes e_2 + ( A^{\ell +2} - A^4 p_\ell )   e_1 \otimes e_2 \otimes e_1
+ A^2 p_\ell  \   e_2 \otimes e_1 \otimes e_1 , \\
e_2 \otimes u_1 & \mapsto & A^\ell \ e_2 \otimes e_1 \otimes e_1 - A^{-4} p_\ell \ u_{13} \\
& & =  p_\ell \ e_1 \otimes e_2 \otimes e_1 + (A^\ell - A^{-2} p_\ell) e_2 \otimes e_1 \otimes e_1 , \\
e_2 \otimes u_2 & \mapsto & A^\ell \ e_2 \otimes e_2 \otimes e_2 , \\
e_2 \otimes u_3 & \mapsto & A^\ell \ e_2 \otimes e_1 \otimes e_2 + A^{\ell +2}\  e_2 \otimes e_2 \otimes e_1
- A^{-2} p_\ell  \ u_{12} \\
& & = p_\ell\  e_1 \otimes e_2 \otimes e_2 + (A^\ell - A^{-2} p_\ell ) e_2 \otimes e_1 \otimes e_2
+ A^{\ell +2} \ e_2 \otimes e_2 \otimes e_1 .
\end{eqnarray*}
Denote these images of $e_i \otimes u_j$ by $w_{ij}$, $i=1,2, j=1,2,3$. 
Since $R^\ell \otimes \mathbb1$ is bijective, its image $\{ w_{ij} \}$ forms a rank 6 submodule.

	Since we have
			\begin{eqnarray*}
				 \rank (E(B_1,A) + E(B_3,A)) &=& \rank E(B_1,A) + \rank E(B_3,A) - \rank (E(B_1,A)\cap E(B_3,A)), 
			\end{eqnarray*}
			where $\rank E(B_1,A) = \rank E(B_3,A) = 6$, we find that $\rank (E(B_1,A)\cap E(B_3,A)) \geq 4$, since $8 = \rank V^{\otimes 3}  \geq \rank (E(B_1,A) + E(B_3,A))$. 
			
			We determine the intersection by direct computations. For a vector $x$ to be an element of the intersection, we  find coefficients $\alpha_{ij}$ and $\beta_{mn}$ such that 
			\begin{eqnarray*}
					\sum_{i=1}^2\sum_{j=1}^3 \alpha_{ij} e_i\otimes u_j &=& \sum_{m=1}^2 \sum_{n=1}^3 \beta_{mn} w_{mn}.  
			\end{eqnarray*} 
			Expanding the vectors, and equating each component for each $e_a\otimes e_b\otimes e_c$, we obtain the system for the coefficients $\alpha_{ij}$ and $\beta_{mn}$ as follows, 
			where the left column  $(abc)$ indicates the coefficients of $e_a\otimes e_b\otimes e_c$.
$$\begin{array}{lrcl}
(111) & \alpha_{11} & = & A^\ell \beta_{11} \\
(112) & \alpha_{13} & = & A^\ell \beta_{13} \\
(121) & A^2 \alpha_{13} & = & (A^{\ell + 2} - A^4 p_\ell ) \beta_{13} + p_\ell \beta_{21}   \\
(122) & \alpha_{12} & = &  (A^\ell  - A^2 p_\ell ) \beta_{12} + p_\ell \beta_{23}   \\
(211) & \alpha_{21} & = & A^2 p_\ell  \beta_{13}  + (A^\ell  - A^{-2} p_\ell ) \beta_{21} \\
(212) & \alpha_{23} & = &  p_\ell  \beta_{12}  + (A^\ell  - A^{-2} p_\ell ) \beta_{23} \\
(221) & A^2 \alpha_{23} & = & A^{\ell+2}  \beta_{23} \\
(222) & \alpha_{22} & = & A^\ell \beta_{22}
\end{array}
	$$	
We have that $\alpha_{ij}$   are written by $\beta_{ij}$ alone for $(i,j)=(1, 1), (1,3), (2,2) $ and $ (2,3)$.
The equations for $(121)$ and $(212)$
yield
$$\begin{array}{lrclrcl}
(121) & A^{\ell +2}  \beta_{13} & = & (A^{\ell + 2} - A^4 p_\ell ) \beta_{13} + p_\ell \beta_{21} , & \beta_{21}&=&A^4 \beta_{13} ,  \\
(212) & A^\ell \beta_{23} & = &  p_\ell  \beta_{12}  + (A^\ell  - A^{-2} p_\ell ) \beta_{23},
& \beta_{12} &=& A^{-2} \beta_{23} . \\
\end{array} $$
The remaining equations $(122)$ and $(211)$ provide 
$\alpha_{12} = A^{\ell-2} \beta_{23}$ and $\alpha_{21}=A^{\ell + 4} \beta_{13}$. 
To find a basis, we set 
$\beta_{11}=\beta_{22}=\beta_{13}=\beta_{23}=1$ and find 
that the intersection $E(B_1, A) \cap E(B_3, A)$ is spanned by 
$\{ w_{11}, w_{22}, w_{13}, w_{23} \}$. 
%Since the rank of the intersection is at least 4, this provides the basis for the intersection.
Since these terms contain unique basis 
$e_1 \otimes e_1 \otimes e_1$, $e_2 \otimes e_2 \otimes e_2$, $e_2 \otimes e_1 \otimes e_1$, 
and $e_2 \otimes e_2 \otimes e_1$, respectively, they form a rank 4 submodule, which agrees with the computation
$\rank (E(B_1,A)\cap E(B_3,A)) \geq 4$. 

We show that $E_{(\vec{B}_{\ell_1}, A)}\neq E_{(\vec{B}_{\ell_2}, A) }$ if $\ell_1 \neq \ell_2$. 
We note that for this statement we identify the whole module  $V^{\otimes n} $ that contains these eigenspaces for 
the two surface braids in question. This is done by specifying the basis $(e_i)$ of $V$ that determines basis of $V^{\otimes n} $ as in computations below. 

Set 
$$E_{(\vec{B}_{\ell_1}, A)} = \sum_{(i,j)} \gamma^1_{ij} w_{ij}  \quad {\rm and} \quad  E_{(\vec{B}_{\ell_2}, A)} =\sum_{(i,j)} \gamma^2_{ij} w_{ij}  . $$
Assume that $\ell_1 \neq \ell_2$ and $E_{(\vec{B}_{\ell_1}, A)} = E_{(\vec{B}_{\ell_2}, A)}$.
Then for $w_{13}$, 
 by comparing the coefficients of $e_a \otimes e_b \otimes e_c$ for $(a,b,c)=(1,1,2)$ and $(2,1,1)$, we obtain 
$$
\begin{array}{lrcl}
 (112) &A^{\ell_1} \  \gamma^1_{13} & = & A^{\ell_2} \  \gamma^2_{13} , \\
%(121) &  \gamma^1_{13}  (A^{\ell_1 +2} - A^4 p_\ell) &=& \gamma^2_{13}  (A^{\ell_2 +2} - A^4 p_\ell) ,\\
(211) &   A^2 p_{\ell_1} \ \gamma^1_{13}  &=&  A^2 p_{\ell_2}   \  \gamma^2_{13}  .  
\end{array}
$$
From the equation for $(211)$, we obtain $\gamma^1_{13}=\gamma^2_{13}$. 
If $\gamma^1_{13}=\gamma^2_{13}\neq 0$, then $(112)$ implies $\ell_1=\ell_2$, contradicting the assumption. 
Hence $\gamma^1_{13}=\gamma^2_{13} = 0$. Similarly we find that $\gamma^1_{23}=\gamma^2_{23}= 0$.
Since both  $E_{(\vec{B}_{\ell_i}, A) }$, $i=1,2$, contain  $\langle w_{11}, w_{22} \rangle$, we have 
$E_{(\vec{B}_{\ell_1}, A) } \cap  E_{(\vec{B}_{\ell_2}, A) }= \langle w_{11}, w_{22} \rangle$.
Thus we obtain that $E_{(\vec{B}_{\ell_1}, A)}\neq E_{(\vec{B}_{\ell_2}, A) }$. 

For $\lambda=-A^{-3}$, using Lemma~\ref{lem:VEA-3} one computes the images of $(R^\ell \otimes \mathbb 1)(V \otimes E_{-A^{-3}})$
as follows:
\begin{eqnarray*}
e_1 \otimes u_4 &\mapsto & A^\ell(  -A^2 e_1 \otimes e_1 \otimes e_2 + e_1 \otimes e_2 \otimes e_1 ) 
+ p_\ell ( -A^2 e_1 \otimes e_2 \otimes e_1  + e_2 \otimes e_1 \otimes e_1 ) \\
& & = -A^{\ell +2} \ e_1 \otimes e_1 \otimes e_2
+ (A^\ell - A^2 p_\ell ) e_1 \otimes e_2 \otimes e_1  + p_\ell \ e_2 \otimes e_1 \otimes e_1 =: w_1   , \\
 e_2 \otimes u_4 &\mapsto & A^\ell( -A^2 e_2 \otimes e_1 \otimes e_2 + e_2 \otimes e_2 \otimes e_1 ) 
 + p_\ell ( -A^2 e_1 \otimes e_2 \otimes e_2  + e_2 \otimes e_1 \otimes e_2 ) \\
 & & = -A^2 p_\ell \ e_1 \otimes e_2 \otimes e_2 + (- A^{\ell +2} + p_\ell) e_2 \otimes e_1 \otimes e_2 
 + A^\ell \ e_2 \otimes e_2 \otimes e_1 =: w_2 .
\end{eqnarray*}
We set $\sum_{i=1,2} \alpha_i e_i \otimes u_4 = \sum_{m=1,2} \beta_m w_m$ to obtain for each basis terms:

$$\begin{array}{lrcl}
(112) & -A^2 \alpha_1 & = & - A^{\ell +2} \beta_{1} \\
(121) &  \alpha_{2} & = & (A^{\ell } - A^2 p_\ell ) \beta_{1}  \\
(122) &  0 & = &  - A^2 p_\ell \  \beta_{2}   \\
(211) & 0 & = & p_\ell \beta_1 \\
(212) & - A^2 \alpha_{2} & = & (- A^{\ell +2} + p_\ell ) \beta_{2}  \\
(221) & \alpha_2 &=& A^\ell \beta_2
\end{array}
	$$		
and we obtain $\alpha_i$ and $\beta_m$ all vanish, hence $E_{(\vec{B}_\ell, - A^{-3})}$ is trivial. 
\end{proof}

It is also desirable to have examples using 
the {\it adjoint YBOs} \cite{CCES-adjoint} for Hopf algebras that correspond to the Hopf algebra version of the conjugation,  
and 
to expand examples to  ribbon surface braids and
  twist spun knots. 

\section{Quantum 3-cocycle invariant for surface braids}\label{sec:q3}

In this section we define quantum 3-cocycle invariants for surface braids.
We use the Yang-Baxter (YB) cohomology developed in \cite{SZ-YBH}, in particular its 3-cocycles, which we review briefly.

\begin{figure}[htb]
	\begin{center}
		\includegraphics[width=3in]{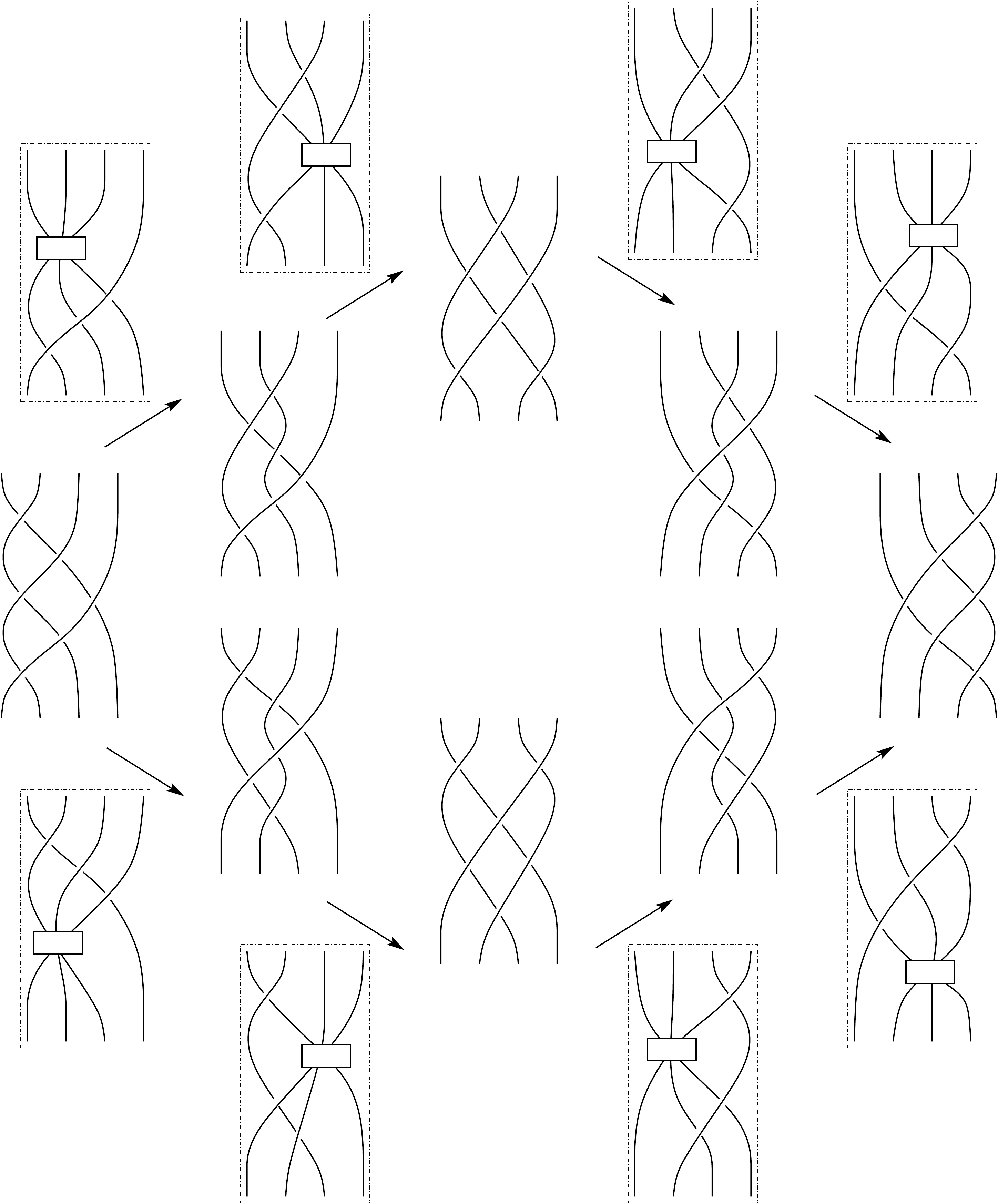}
	\end{center}
	\caption{ The tetrahedral move and the YB 3-cocycle condition
	}
	\label{tetra}
\end{figure}

In Figure~\ref{tetra} from~\cite{SZ:BC}, the tetrahedral move is depicted, that is one of the moves in movie descriptions of knotted surfaces~\cite{CS:book}.  It also indicates 
 the  relation to the YB 3-cocycle condition~\cite{SZ-YBH}.
From the leftmost to the rightmost braid form, there are two ways to transform them successively by a sequence of 
a braid relation as depicted, and the moments of transition is schematically depicted by boxes that represent the places where the braid relation is performed. 
When the box is interpreted as elements of ${\rm Hom}(V^{\otimes 3} , V^{\otimes 3} )$, the equation obtained by equating the sum of the top 4 terms and  the sum of the bottom  coincides with the YB 3-cocycle condition. 
 We will use the shorthand notation for $ \psi \in {\rm Hom}(V^{\otimes 3} , V^{\otimes 3} ) $ 
\begin{eqnarray*}
				\psi_{(i)} &=& \mathbb 1^{i-1}\otimes \psi \otimes \mathbb 1^{n-i-3}. 
\end{eqnarray*}
The fist term of the 3-cocycle condition is depicted with maps labeled in Figure~\ref{tetra1}.
The 3-cochain $\psi$ is placed at the three strands from the left, indicated by the subscript $(1)$. Since the matrix multiplication is from the left, $\psi_{(1)}$ is placed at the rightmost position, followed by 
$ R_{(1) } R_{ (2) }   R_{(3)}  $, the three YBOs placed at the first, second and third tensor factors from the bottom,
giving rise to the first term   $R_{(1) } R_{ (2) }   R_{(3)}  \psi_{(1)}  $. Then the YB 3-cocycle condition is written as

\begin{figure}[htb]
	\begin{center}
		\includegraphics[width=.8in]{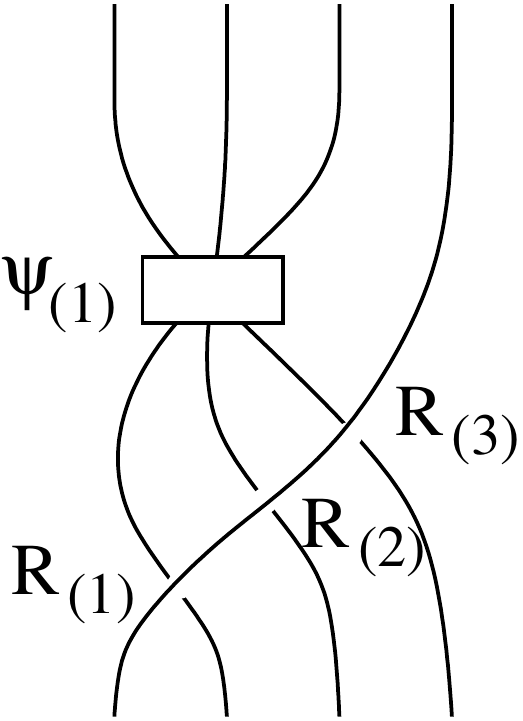}
	\end{center}
	\caption{ The first term of the Yang-Baxter 3-cocycle condition
	}
	\label{tetra1}
\end{figure}

%Specifically it is written as 
\begin{eqnarray}
\begin{aligned}
 & &  R_{(1) } R_{ (2) }   R_{(3)}  \psi_{(1)}  
 + R_{(1) }  \psi_{(2)} R_{(1) }  R_{ (2) }  
 +    R_{(3)}    \psi_{(1)}      R_{(3)}  R_{ (2) }   
 +    R_{(3)}  R_{ (2) }   R_{(1) }  \psi_{(2)}   \label{eqn:tetra}  \\
 &=&    
  \psi_{(1)}    R_{(3)}   R_{ (2) }  R_{(1) }  
 +    R_{(2)}    R_{ (1) }    \psi_{(2 )}  R_{(1) } 
+  R_{(2)}    R_{ (3) }    \psi_{(1)}  R_{(3) } 
+   \psi_{(2)}    R_{(1)}   R_{ (2) }  R_{(3) }  ,
\end{aligned}
\end{eqnarray}
where $\psi \in  {\rm Hom}(V^{\otimes 3} , V^{\otimes 3} )$ is a YB 3-cochain, and the subscripts indicate the tensor factors where these operators are placed as before.
The group of YB 3-cocycles is denoted by $Z^3(V, V)$.

To define the quantum cocycle  invariant using YB 3-cocycles, first we define color vectors for compelementary regions of charts, in an approach analogous to the idea for the quandle 3-cocycle invariant~\cite{CJKLS}. 

Let $S$ be a surface braid and $\Gamma (S)$ be a  chart of $S$ as before, with a base point $q_0$. 
Let ${\rm R}\Gamma (S)=\{ G_i \}_{i=1}^h$ be the set of connected components of the complementary regions $D^2 \setminus \Gamma(S)$. Pick and fix a point $q_i \in G_i$ for each $i=1, \ldots, h$. 
Pick a path $\gamma_i$ from $q_0 $ to $q_i$ in $D^2$ that misses black and white vertices of $\Gamma(S)$,
 and  in general position with $\Gamma(S)$. 
As one traces $\gamma_i$ from $q_0 $ to $q_i$, one encounters transverse double 
 intersection points 
$p_i^{(j)}$, $j=1, \ldots, h_i$, with edges of $\Gamma (S)$, that are labeled by braid generators or their inverses.
Let $\beta_i^{(j)}$ be such a sequence, and set 
 $\beta_i=\beta_i^{(h_i)} \cdots  \beta_i^{(1)}  \in \mathbb B_n$.

\begin{sloppypar}
Let $R$ be a YBO over $V$ with coefficient ring $\mathbb k$. 
Apply the representation $\rho : \mathbb B_n \rightarrow {\rm Hom}(V^{\otimes n} , V^{\otimes n})$ using $R$ for braid generators, and let 
 $ D_i := \rho ( \beta_i ) $. 
 For $v \in V$, let ${\mathcal C} (G_i, v) := D_i v$ and call it a {\it color vector} of the region $G_i$. 
 \end{sloppypar}

\begin{lemma}\label{lem:paths}
	Suppose $v \in E_{( \vec{B}, 1)} $ is  an eigenvector of eigenvalue $1$.
	Then the  {\emph color vector}  ${\mathcal C} (G_i, v ) = D_i v$ does not depend on a choice of  paths $\gamma_i$.
\end{lemma}

\begin{proof}
	Since the composition $\alpha^{-1} \beta$ of two paths with the same end points forms a loop, it is sufficient to prove that $Dv=v$ for an operator $D$ corresponding to a loop $\gamma$ based at $q_0$. For such a loop, we first consider the effects of sliding through $4$-valent and $6$-valent (white) vertices.  At a $4$-valent vertex, in terms of braid generators, we have an equality of type 
	$\sigma_j \sigma_k=\sigma_k \sigma_j$ or their inverse variants, for $|j-k|>1$. Therefore, sliding the loop $\gamma$ through such a vertex corresponds to swapping the order of $R_{(j)}$ and $R_{(k)}$ for $|j-k|>1$. The corresponding operator $D$ associated to $\gamma$ is unchanged, since $R_{(j)}$ and $R_{(k)}$ commute by construction. This situation is depicted in Figure~\ref{homotopy} (A) and (B), where 
	the dotted arcs represent the arc $\gamma$ before and after the homotopy.

\begin{figure}[htb]
	\begin{center}
		\includegraphics[width=3.5in]{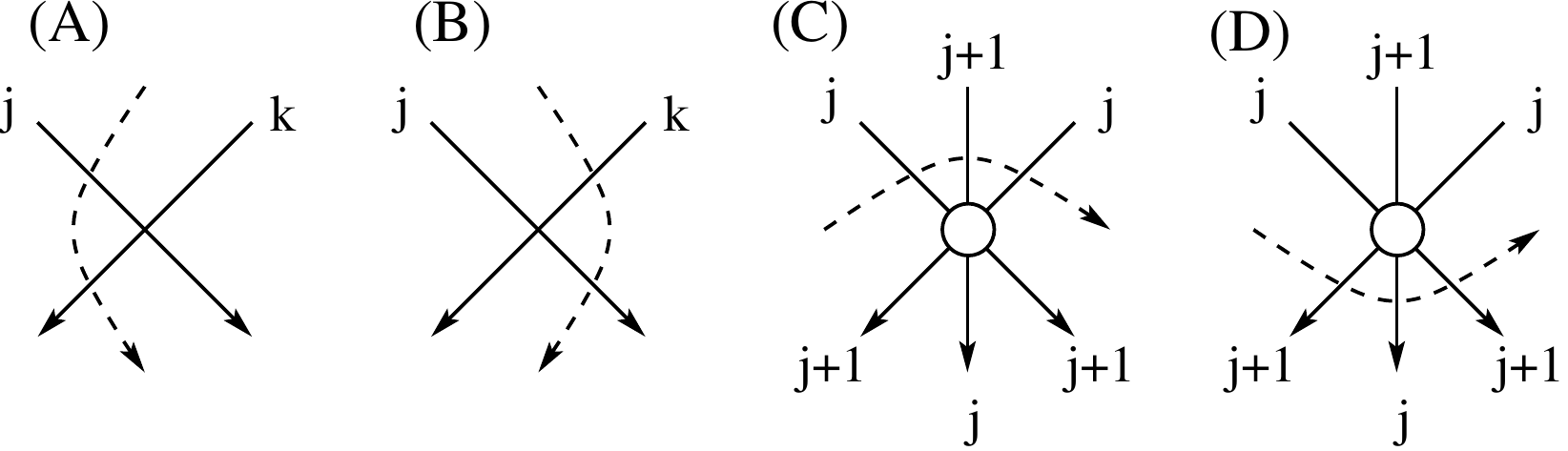}
	\end{center}
	\caption{ Homotopies through a crossing and a white vertex
	}
	\label{homotopy}
\end{figure}

	Likewise, the effect of sliding $\gamma$ through a white vertex is that of performing a 
	braid relation 
	$\sigma_j \sigma_{j+1} \sigma_j=\sigma_{j+1} \sigma_j \sigma_{j+1}$, or the associated inverse variants. Since $R$ satisfies the YBE,  performing a slide of $\gamma$ through a white vertex does not change the corresponding operator $D$ associated to $\gamma$. This homotopy is depicted in Figure~\ref{homotopy} (C) and (D). 
	
	When a pair of intersection points of $\gamma$ and an edge of $\Gamma(S)$ is canceled or added, 
	a canceling pair $\sigma_j \sigma_j^{-1}$ or $\sigma_j^{-1} \sigma_j$
	is added/canceled. The corresponding pair as a subsequence of 
	$D$ 
	cancels, and the operator 
	  $D$  is invariant under this change of braid word. 
	
	After performing finitely many of the previous moves, we obtain a new loop, which we still denote by $\gamma$ with an abuse of notation, whose corresponding operator $D$ is equal to the initial one because it is obtained by swapping commuting elements, or applying the YBE, or inserting/deleting 
	 cancelling pairs. The only points of the chart $\Gamma(S)$ where we  have not freely homotoped $\gamma$ off  
     are the black vertices. However, we can  homotope and  arrange $\gamma$ in such a way that it factors in a product $(a_{i_1}^{-1}s_{i_1}^{\epsilon_{i_1}}a_{i_1})\cdots (a_{i_k}^{-1}s_{i_k}^{\epsilon_{i_k}}a_{i_k})$ of loops, each obtained by taking a Hurwitz arc $a_{i_r}$, a loop $s_{i_r}$ around the $i_r$th black vertex taken with positive or negative sign, and concatenating with the inverse of the Hurwitz arc $a_{i_r}$. 
	This corresponds to a product of braids of type $(w_{i_1}^{-1} \sigma^{\epsilon_{i_1}}_{i_1} w_{i_1} ) \ldots (w_{i_k}^{-1} \sigma^{\epsilon_{i_k}}_{i_k} w_{i_k}) $. Therefore, we have 
              $D= B_{i_1}^{\epsilon_{i_1}}\cdots B_{i_k}^{\epsilon_{i_k}}$. 
	Since $v$ is in the $1$-eigenspace of each $B_{i_r}$, and therefore of each $B_{i_r}^{-1}$ as well, we obtain $Dv = v$ as required. 
\end{proof}

\begin{figure}[htb]
\begin{center}
\includegraphics[width=2.2in]{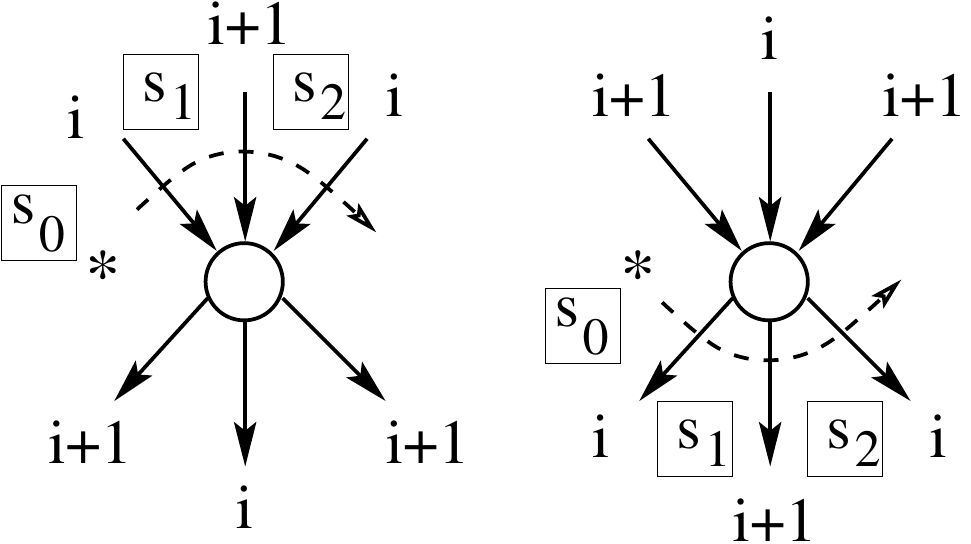}
\end{center}
\caption{Source regions and signs at white vertices}
\label{source}
\end{figure}

To define the 3-cocycle invariant, a few technical definitions needed for construction of the invariant follow.

\begin{definition}
{\rm
For a white vertex $T$, there are 6 adjacent complementary regions, and each region is bordered by two edges.
There is a unique region among the 6 regions such that the arc crossing the adjacent edges starting from that region goes from right to left across both edges, so that both arcs in this direction receive 
positive braid generators as labels. 
This region is called the {\it source region} of the vertex, and is denoted by $G_T$. 
Source regions for two possible labels are indicated by the symbol $*$ in Figure~\ref{source}. 
%Other regions have one positive and one negative, or both negative generators as labels.

When one crosses three arcs that are incoming to the vertex from the source region, 
there are two possibilities: (1) the labels are $(\sigma_i, \sigma_{i+1}, \sigma_i)$ in this order, or (2) $( \sigma_{i+1}, \sigma_i,  \sigma_{i+1})$. The case (1) is defined to be a {\it positive} white vertex, and in case (2) {\it negative}. 
They are depicted in Figure~\ref{source} on the left and right, respectively.

}
\end{definition}

In the figure, the dotted arcs are used later. For a positive case (left) the dotted arc is clockwise and agrees with the order $(\sigma_i, \sigma_{i+1}, \sigma_i)$, and for the negative case (right) the dotted arc is counterclockwise 
and the labels it encounters also agree with $(\sigma_i, \sigma_{i+1}, \sigma_i)$.

\begin{definition}
{\rm
A  Yang-Baxter 3-cocycle  $\psi \in Z^3(V,V)$  with respect to a given YBO $R$
is said to have (or satisfy)
the {\it quandle condition }  if it satisfies 
$$\psi (v \otimes w)=0, $$
whenever either $v$ or $w$  (exclusively)  is a $1$-eigenvector of $R$. }
\end{definition}

	Note that if $v$ is a 1-eigenvector of $R$, then $v \in V^{\otimes 2}$, and then $w \in V$, so that $v \otimes w \in V^{\otimes 3}$, 
	the domain of $\psi$.

\begin{definition}
{\rm
Let $\Gamma(S)$ be a chart with base point $q_0$.
Let $v \in E_{(\vec{B}, 1)}$. 
Let $W\Gamma(S)$ be the set of white vertices, and  for $T \in W\Gamma(S)$, let $s(T,v)$ be the region color 
vector 
 of the source region  at $T$ with respect to $v$ assigned to $q_0$. 
 This is called the {\it source region color vector} at $T$. 
}
\end{definition}

\begin{definition}[\cite{CCES-coalgebra}]
{\rm 
A map $\varepsilon: V \longrightarrow \mathbb k$ such that $(\varepsilon \otimes \varepsilon )R = \varepsilon \otimes \varepsilon$ is called a {\it Yang-Baxter counit}, or a YB counit for short. 
%In what follows we assume that there exists a map $\varepsilon: V \longrightarrow \mathbb k$ such that $(\varepsilon \otimes \varepsilon )R = \varepsilon \otimes \varepsilon$. Such map $\varepsilon$ satisfying this %braided 
%condition will be referred to as a {\it Yang-Baxter counit}, or YB counit for short. 
}
\end{definition}

In what follows we assume that a YBO has a YB counit.

\begin{example}\label{ex:quandleunit}
	{\rm 
			As an example of a YBO endowed with a YB counit, we consider any set-theoretic YBO over a set $X$. We define $R$ to be the linearized operator on the free vector space 
			$V=\mathbb k [X]$, %$V = \mathbb k\langle X\rangle$, 
			and $\varepsilon : V \longrightarrow \mathbb k$ given by $\varepsilon(x) = 1$ for each $x\in X$ and extended by linearity to $V$. 
			Then $\varepsilon$ is a YB counit. This is seen by computation. Let the operation on $X$, $r: X \times X \rightarrow X \times X$,  be expressed by 
$(x, y) \mapsto (r_1(x, y), r_2 (x,y))$  for $x, y \in X$.
Then for a  basis vector $x \otimes y \in V \otimes V$, one computes
$$ (\varepsilon \otimes \varepsilon) R (x \otimes y)= 	(\varepsilon \otimes \varepsilon) (r_1(x, y)\otimes  r_2 (x,y))
= 1 = (\varepsilon \otimes \varepsilon)(x \otimes y) . $$ 
			
				In particular, any quandle gives rise to such YBO with corresponding YB counit, following the construction of Section~\ref{sec:quandle_col}. 
	}
\end{example}

\begin{example}[\cite{CCES-coalgebra}]
	{\rm  
			Let $H$ be a Hopf algebra over a ground ring $\mathbb k$. 
						The Hopf algebra version of group conjugation 
			$x\otimes y \mapsto  S(y^{(1)})xy^{(2)}$ called the adjoint induces a YBO called the adjoint 
			YBO % R-matrix.
						\begin{eqnarray*}
					x\otimes y \mapsto y^{(1)}\otimes S(y^{(2)})xy^{(3)}, 
			\end{eqnarray*}
			where $y^{(i)}$ is Sweedler's notation for the comultiplication,
				 and $y^{(1)} \otimes y^{(2)} \otimes y^{(3)} $ represents 
			$ (\Delta \otimes \mathbb 1 ) \Delta (y) =  (\mathbb 1   \otimes \Delta )   \Delta(y) $ for the comultiplication $\Delta$. The concatenation of symbols represents multiplication. 
 If $\varepsilon$ indicates the counit of the Hopf algebra, one has that 
			\begin{eqnarray*}
					(\varepsilon \otimes \varepsilon )R(x\otimes y) &=& \varepsilon(y^{(1)}) \varepsilon(S(y^{(2)})xy^{(3)})\\
					&=& \varepsilon(x)  \varepsilon(y^{(1)}S(y^{(2)})y^{(3)})\\
					&=& \varepsilon(x)\varepsilon(y)\\
					&=& (\varepsilon \otimes \varepsilon )(x\otimes y). 
			\end{eqnarray*}
			Therefore, it follows that $\varepsilon$ is a YB counit with respect to the adjoint YBO.
			%induced by the quantum quandle operation. 
	}
\end{example}

\begin{example}
{\rm

The bracket YBO defined by $R=A \mathbb 1 + A^{-1} \cap \cup$ as in Section~\ref{sec:bracket} does not have 
%such
 a YB counit as above unless $A=1$. 
This is seen by computation as follows. If such an $\varepsilon $ exists, then one computes, with $i=\sqrt{-1}$, 
\begin{eqnarray*}
(\varepsilon\otimes \varepsilon)R(e_1 \otimes e_2) 
& = & 
A \varepsilon( e_1) \varepsilon( e_2) 
+ A^{-1} (\varepsilon\otimes \varepsilon) [ (  i A ) ( iA e_1 \otimes e_2 - iA^{-1} e_2 \otimes e_1 )] \\
&=& 
A  \varepsilon( e_1)  \varepsilon( e_2) + i [ iA  \varepsilon( e_1)\varepsilon( e_2)  - iA^{-1}  \varepsilon( e_2)  \varepsilon( e_1) \\
&=& A^{-1} \varepsilon( e_1)  \varepsilon( e_2) 
\end{eqnarray*}
as claimed.
}
\end{example}

For a white vertex $T\in W\Gamma(S)$ and a given vector $v\in E(\vec B,1)$, 
let 
$s(T, v)$ %to 
be the color vector assigned to the source region $G_T$ of $T$ as indicated by $*$ in Figure~\ref{source}. Let $\gamma_T$ denote the path connecting $q_0$ and $G_T$. Let $B_T$ denote the product of matrices associated to $\gamma_T$, obtained by composing the operators $R^\pm_{(k)}$ associated to the generators $\sigma_k^{\pm}$ along the path $\gamma_T$ from 
$q_0$ to $G_T$.  Then we have $s(T,v)=B_T v$. 

 Let $i$ denote the index of the braid relation that is applied at $T$, i.e. suppose that at $T$ we have $\sigma_{i}\sigma_{i+1}\sigma_i= \sigma_{i+1}\sigma_i \sigma_{i+1}$
 for a positive white vertex as in the left of Figure~\ref{source}. 
% {(Seems not used:) Then prolong  $B_T$ along the white vertex $T$ as $R_{(i)}R_{(i+1)}R_{(i)}B_T$ 
% along the dotted line. }
 
 \begin{definition}
	{\rm 

 Let $\psi \in Z^3(V,V)$ be a Yang-Baxter 3-cocycle with respect to a given YBO $R$ that satisfies the quandle condition. 

For each $v\in E(\vec B,1)$ and $T\in W\Gamma(S)$, 
 we define the Boltzmann weight $BW_\psi(T,v)$ as 
\begin{eqnarray*}
		BW_\psi(T,v) &=&  \epsilon(T)\varepsilon^{\otimes n} \psi_{(i)}  (s(T,v) ) \ = \    \epsilon(T)\varepsilon^{\otimes n}\psi_{(i)}B_{T} v. 
\end{eqnarray*}

			%With notation as given above, 
			Then we define the map $\Psi(S,\psi) : E(\vec B,1) \longrightarrow \mathbb k$ as
			\begin{eqnarray*}
					\Psi(S,\psi)(v) &=& \sum_{T\in W\Gamma(S)} BW_\psi(T,v). 
			\end{eqnarray*}
	}
\end{definition}

\begin{theorem}\label{thm:3cocy}
The map $\Psi(S, \psi)$ is indeed an invariant of surface braids.
\end{theorem}

\begin{proof}
We show that, for any $v \in E(\vec B,1)$, $\Psi(S, \psi)(v)$ remains invariant under equivalence of surface braids.
Specifically, invariance under slide equivalence and chart moves 
(see below)  
are examined.
 It can be assumed that a 
 slide equivalence does not alter charts, and invariance of $E_{(\vec B, 1)}$ was shown in Theorem~\ref{thm:sfc_braid}.
Chart moves is a  set of moves of charts that generate equivalence of charts~\cite{Kamada:SB}, and they are  depicted in Figure~\ref{Cmoves}. The upper half of moves do not involve white vertices, and the lower moves do. 
Thus we check the lower chart moves (A) through (D).

\begin{figure}[htb]
\begin{center}
\includegraphics[width=4in]{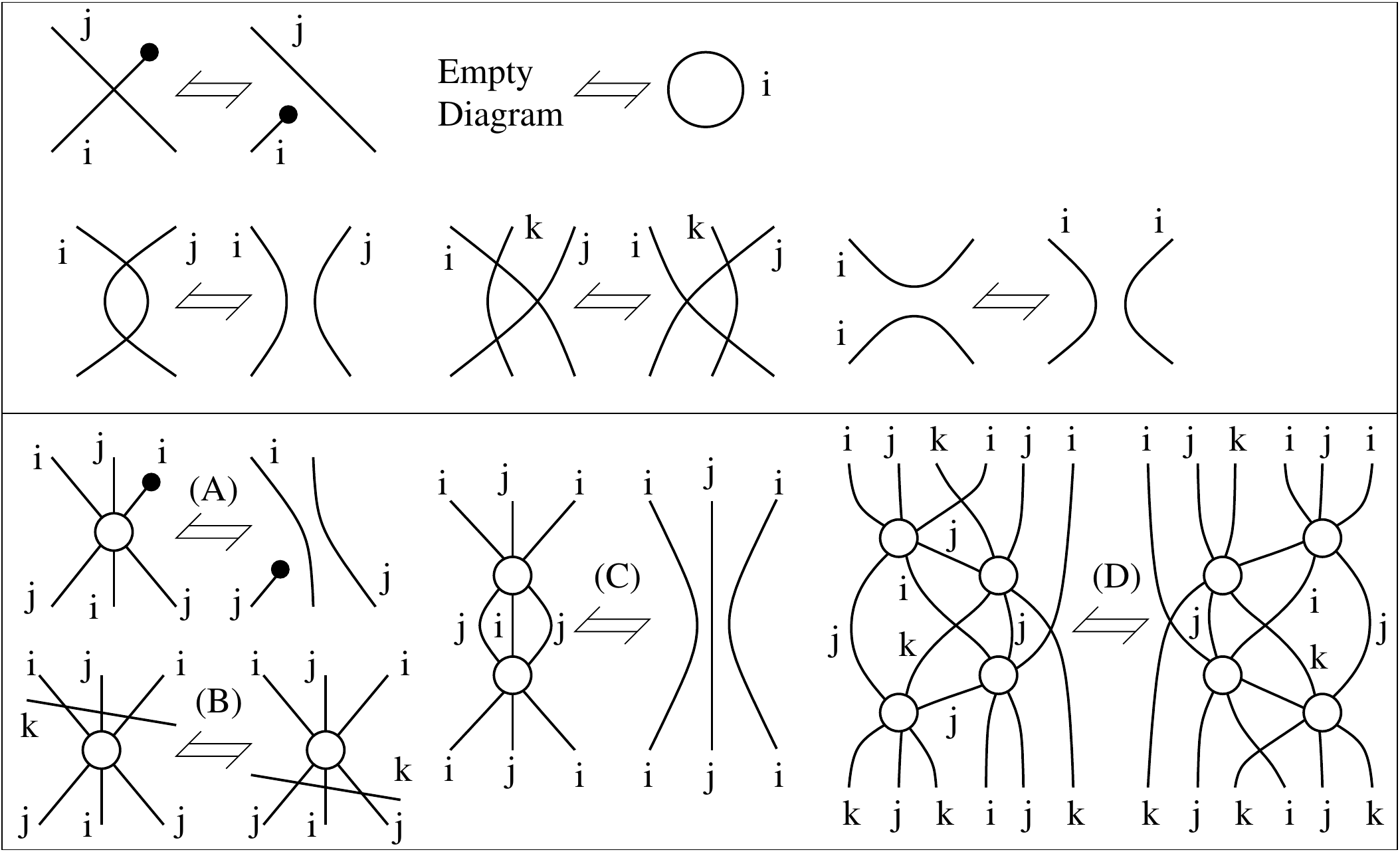}
\end{center}
\caption{Chart moves of surface braids}
\label{Cmoves}
\end{figure}

First we prove the invariance under  the move of a black vertex going through a white vertex, as shown in Figure~\ref{Cmoves}  (A).
We consider the case where the labels in the left-hand-side of (A) are $i$, $i+1$ and $i$  as in the figure, and they are all 
 incoming direction into the white vertex $T$. The other cases are  analogous. 
 We denote by $w_h$ the braid word associated to a Hurwitz arc $a_h$ connecting the base point $q_0$ 
 to the black vertex $y_h$ appearing in the 
  right  figure. The corresponding braid word in the braid system is given by $w_h^{-1} \sigma_i^{-1} w_h$, which is represented by the operator 
 $\rho(w_h)^{-1} R_{(i)}^{-1} \rho(w_h)$.
 The inverse generator $\sigma_i^{-1}$ for the black vertex is due to that the edge is outgoing from the black vertex.
 
 The image under a representation $\rho$ of a  Hurwitz arc for the black vertex in the left figure is $ \rho(w_h)$.
 Since $v\in E(\vec b, 1)$, it follows that $\rho(w_h)^{-1} R_{(i)}^{-1} \rho(w_h) v = v$, and therefore that $R_{(i)}^{-1} \rho(w_h) v = \rho(w_h) v$. Hence we have that 
$$\rho(w_h) v\in E(R_{(i)},1) = V^{\otimes (i-1)} \otimes E(R,1)\otimes V^{\otimes n-i-3}. $$
From the left figure, we see that a path $\gamma_T$ connecting the base point to the region $G_T$ of the white vertex $T$ is homotopic to a path that intersects all the same edges as the Hurwitz arc $a_h$. Hence the 
 contribution of the white vertex $T$ to the map $\Psi(S,\psi)$ for $v\in E(\vec b,1)$ is given by 
$\epsilon(T)\varepsilon^{\otimes n}\psi_{(i)}\rho(w_h)v = 0$, since $\psi_{(i)}$ is trivial on $V^{\otimes (i-1)} \otimes E(R,1)\otimes V^{\otimes n-i-2}$ by the quandle condition of $\psi$. This shows that removing the white vertex as in (A) does not change the value of $\Psi(S,\psi)$ for  any $v\in E(\vec b,1)$, as required.

\begin{figure}[htb]
\begin{center}
\includegraphics[width=2.3in]{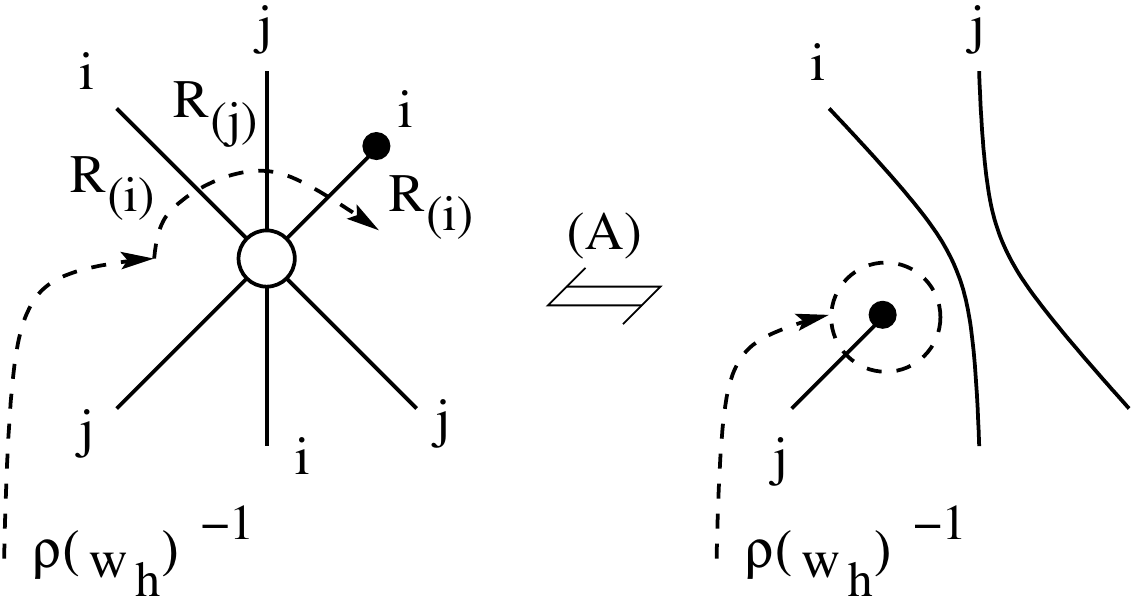}
\end{center}
\caption{Checking chart move (A)}
\label{CmoveA}
\end{figure}

Next we show the invariance under move (B) in Figure~\ref{Cmoves}. 
Let $v$ be the color vector of the source region $G_T$ of the white vertex $T$ in the  left-hand side of the figure.
We may assume that the arc $\gamma_T$ to $G_T$ does not intersect the arc labeled $k$ in the left-hand side of the figure, and does intersect in the right-hand side, in the positive direction. 
Then  the corresponding Botlzmann weight at $T$ is computed for the left-hand side as 
\begin{eqnarray*}
		BW_\psi(T,v) &=& \epsilon(T)\varepsilon^{\otimes n}\psi_{(i)}( s(T, v) ), 
\end{eqnarray*}
while for the corresponding term on the right-hand-side we have a weight computed as
\begin{eqnarray*}
	BW_\psi(T,v) &=& \epsilon(T)\varepsilon^{\otimes n}\psi_{(i)}R_{(k)} (s(T,v)). 
\end{eqnarray*}
The situation is depicted in Figure~\ref{CmoveB}. 
Since $|k-i|, |k-i-1| > 1$, we have
\begin{eqnarray*}
		\varepsilon^{\otimes n}\psi_{(i)}R_{(k)}(s(T,v)) &=& 
		\varepsilon^{\otimes 3}\psi_{(i)} (\varepsilon^{\otimes (i-1)}\otimes \mathbb 1^{\otimes 3}\otimes \varepsilon^{\otimes (n-i-3)})R_{(k)}(s(T,v))\\
		&=& \varepsilon^{\otimes 3}\psi_{(i)} (\varepsilon^{\otimes (i-1)}\otimes \mathbb 1^{\otimes 3}\otimes \varepsilon^{\otimes (n-i-3)})(s(T,v))\\
		&=&  \varepsilon^{\otimes n}\psi_{(i)}(s(T,v)), 
\end{eqnarray*}
where we have used the fact that $(\varepsilon \otimes \varepsilon )R = \varepsilon \otimes \varepsilon$. 
This shows that move (B) does not change the corresponding Boltzmann weight.

\begin{figure}[htb]
\begin{center}
\includegraphics[width=3in]{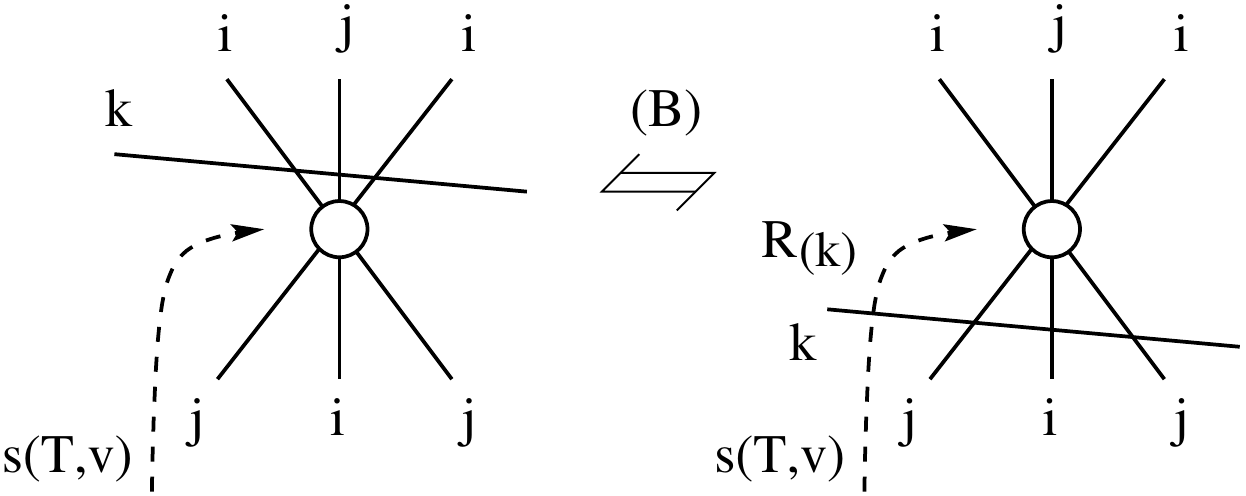}
\end{center}
\caption{Checking chart move (B)}
\label{CmoveB}
\end{figure}

Move (C), corresponding to the cancelation of a pair of white vertices with opposite signs,  does not alter the sum of cocycle values due to the sign convention from which the pair cancels with opposite signs. 

%%%
The YB 3-cocycle condition, with the leftmost index $i$,  reads 
\begin{eqnarray*}
 & &  R_{(i) } R_{ (i+1) }   R_{(i+2)}  \psi_{(i)}  
 + R_{(i) }  \psi_{(i+1)} R_{(i) }  R_{ (i+1) }  
 +    R_{(i+2)}    \psi_{(i)}      R_{(i+2)}  R_{ (i+1) }   
 +    R_{(i+2)}  R_{ (i+1) }   R_{(i) }  \psi_{(i+1)}  \\
 &=&    
  \psi_{(i)}    R_{(i+2)}   R_{ (i+1) }  R_{(i) }  
 +    R_{(i+1)}    R_{ (i) }    \psi_{(i+1 )}  R_{(i) } 
+  R_{(i+1)}    R_{ (i+2) }    \psi_{(i )}  R_{(i+2) } 
+   \psi_{(i+1)}    R_{(i)}   R_{ (i+1) }  R_{(i+2) }  .
\end{eqnarray*}

\begin{figure}[htb]
\begin{center}
\includegraphics[width=2.5in]{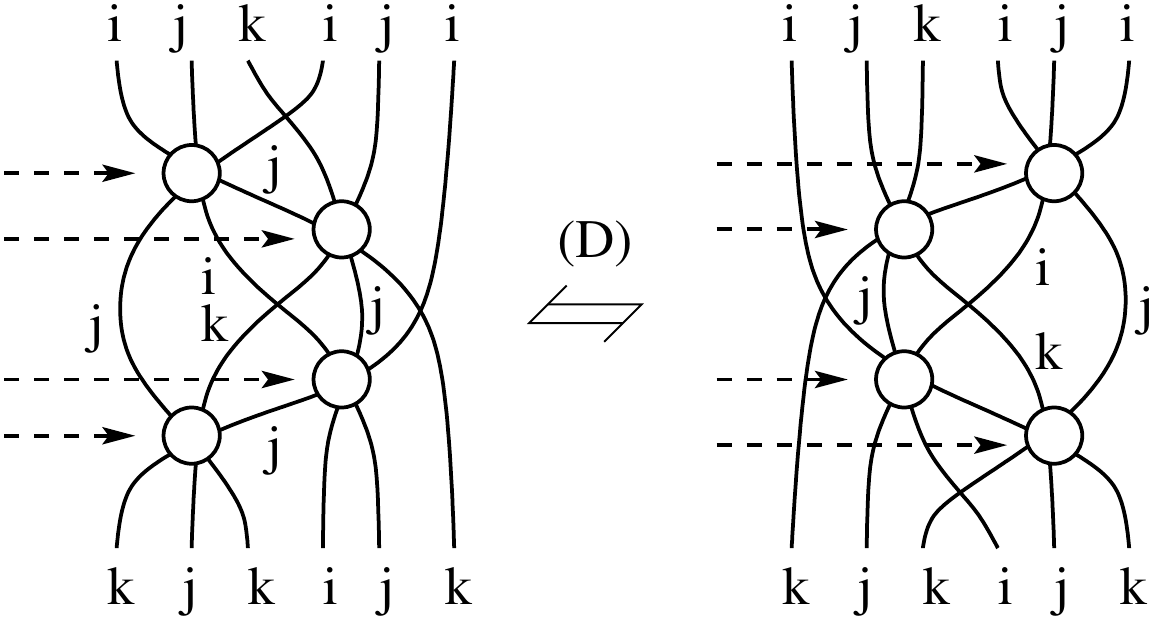}
\end{center}
\caption{Checking chart move (D)}
\label{CmoveD}
\end{figure}

Applying the YB counit $\varepsilon^{\otimes 4}$ on both sides from the left we obtain 
\begin{eqnarray*}
 & & \varepsilon^{\otimes 4}  \psi_{(i)}  
 + \varepsilon^{\otimes 4} \psi_{(i+1)} R_{(i) }  R_{ (i+1) }  
 +   \varepsilon^{\otimes 4}  \psi_{(i)}      R_{(i+2)}  R_{ (i+1) }   
 +  \varepsilon^{\otimes 4} \psi_{(i+1)}  \\
 &=&    
  \psi_{(i)}    R_{(i+2)}   R_{ (i+1) }  R_{(i) }  
 +  \varepsilon^{\otimes 4}   \psi_{(i+1 )}  R_{(i) } 
+  \varepsilon^{\otimes 4}   \psi_{(i )}  R_{(i+2) } 
+  \varepsilon^{\otimes 4} \psi_{(i+1)}    R_{(i)}   R_{ (i+1) }  R_{(i+2) }  .
\end{eqnarray*}
Let $T$ be the top white vertex of the left-hand side of the chart move (D).
Let $i$ be the leftmost label as in Figure~\ref{Cmoves}.
Then from the figure we find that the contributions of these four white vertices on the left and right sides are 
\begin{eqnarray*}
 & & \varepsilon^{\otimes 4}  \psi_{(i)}  s(T,v)
 + \varepsilon^{\otimes 4} \psi_{(i+1)} R_{(i) }  R_{ (i+1) }  s(T, v) \\ %B_T (v) \\
 & & 
 +   \varepsilon^{\otimes 4}  \psi_{(i)}      R_{(i+2)}  R_{ (i+1) }  s(T, v)  % B_T (v)
 +  \varepsilon^{\otimes 4} \psi_{(i+1)} s(T, v) \\ %B_T (v) \\
 &=&    
  \psi_{(i)}    R_{(i+2)}   R_{ (i+1) }  R_{(i) }   s(T, v) % B_T (v) 
 +  \varepsilon^{\otimes 4}   \psi_{(i+1 )}  R_{(i) }  s(T, v) \\ %B_T (v) \\
& & +  \varepsilon^{\otimes 4}   \psi_{(i )}  R_{(i+2) } s(T, v)  %  B_T (v) 
+  \varepsilon^{\otimes 4} \psi_{(i+1)}    R_{(i)}   R_{ (i+1) }  R_{(i+2) }  s(T, v) % B_T (v)  
\end{eqnarray*}
where four terms on each side correspond to four dotted lines in each side in Figure~\ref{CmoveD}, 
and the invariance under (D) follows. 
\end{proof}

For the rest of the section we examine an example constructed from quandle rings as in Example~\ref{ex:22}.

\begin{lemma} \label{lem:q3cocy}
Let $(X, *)$ be a quandle and $V={\mathbb k}[X]$ be its quandle ring.
Let $\theta \in Z^3_{\rm Q}(X, \mathbb k)$, where the abelian group of $\mathbb k$ is used as the coefficient of quandle cohomology group of $X$. 

%Let $\mathbb k =\mathbb r  [\hbar]/(\hbar^2) $ where the abelian group of a unital ring $\mathbb r$ is used
%for the quandle cohomology groups of $X$.
%Let $\theta \in Z^3_{\rm Q}(X, A)$.

Let $\psi: V^{\otimes 3} \longrightarrow V^{\otimes 3}$ be defined for simple tensors by 
$$ 
\psi (e_x \otimes e_y \otimes e_z) :=
\theta(x,y,z)\  e_z \otimes e_{y*z} \otimes e_{ (x*y)*z }, $$
and linearly extended. 
Then $\psi $ is a YB 3-cocycle, $\psi \in Z^3_{\rm YB} (V,V)$.
\end{lemma}

\begin{proof} 
The 3-cocycle condition in quandle cohomology~\cite{CJKLS} for $\theta$ is written as 
\begin{eqnarray*}
\lefteqn{ 
\theta(x,y,z)+ \theta(x*z, y*z, w) + \theta(x, z, w) 
} \\
& = & 
\theta(x*y, z,w)+ \theta(x,y,w) + \theta(x*w,y*w,z*w) .
\end{eqnarray*}

 A direct computation shows that each summand in Equation~\eqref{eqn:tetra} maps a basis vector $e_x\otimes e_y \otimes e_z \otimes e_w$ to a multiple of $e_w \otimes e_{z*w} \otimes e_{(y*z)*w} \otimes e_{((x*y)*z)*w}$, after possibly applying the self-distributive property of $*$ multiple times. The scalar multiples are given by the quandle cocycle for each term in the left-hand-side of Equation~\eqref{eqn:tetra} as:
 \begin{align*}
 		 \theta(x,y,z),\ \theta(x*z,y*z,w),\ \theta(x,y,w),\ \theta(y,z,w), 
 \end{align*}
 while for the right-hand-side we obtain
  \begin{align*}
 \theta(y,z,w),\ \theta(x*y,z,w),\ \theta(x,y,w),\ \theta(x*w,y*w,z*w). 
 \end{align*}
 Equating each monomial, we obtain the quandle $3$-cocycle condition for $\theta$, which vanishes. 
\end{proof}

Let $(X, *)$ be a quandle and $V = \mathbb C [X]$  its quandle ring. 
As in Example~\ref{ex:22}, we decompose bases of $E(\vec B, 1)$ as follows.
From Lemma~\ref{lem:color}, we have $E(\vec B, 1)=\mathbb C [X^n]^H$, where 
$H=\langle f_1, \ldots, f_m\rangle \leq  \rho ( \mathbb B_n )$, 
 where $\rho : \mathbb B_n \rightarrow {\rm Aut} (V^{\otimes n} ) $ is the representation 
defined above  Lemma~\ref{lem:color}. 
As the basis of $\mathbb C [X^n]^H$, we take the following two types of invariant vectors:
\begin{itemize}
\item The monomials $e_{x_1} \otimes \cdots \otimes e_{x_n}$ that are left invariant under the action of $H$, called {\it  monomial invariant vectors} (MIV). 
%The set of imonomial nvariant  vectors are denoted by ${\rm MIV}$.
\item The nonmonomial orbits formed by actions of $H$ from a monomial, called {\it  nonmonomial invariant vectors} (NIV). 
Such a vector is described by $\sum_{h_j \neq 1}  h_j (e_{y_1} \otimes \cdots \otimes  e_{y_n})$, where 
each $e_{y_1} \otimes \cdots \otimes  e_{y_n}$ is not an invariant vector. 
%The set of  nonmonomial invariant vectors are denoted by ${\rm NIV}$.
\end{itemize}
%
%(End delete) }

\begin{example}\label{ex:2TST}
{\rm 
We consider the quantum 3-cocycle invariant defined through Lemma~\ref{lem:q3cocy}.
The YB 3-cocycle $\psi$ defined in the lemma from a quandle 3-cocycle satisfies the quandle condition, and by Example~\ref{ex:quandleunit}, Theorem~\ref{thm:3cocy} implies that a quantum 3-cocycle invariant is well defined
from this $\psi$. Here we take $\mathbb k=\mathbb F_3$, the field of three elements, whose additive group $\Z_3$ is identified with the coefficient group of quandle cohomology, so that $\theta$ in Lemma~\ref{lem:q3cocy} takes values in $\mathbb k$. 
We compare this invariant with the original quandle cocycle invariant for the following example. 
In \cite{CJKLS}, the quandle cocycle invariant $\Psi_X(\hat S)$ is computed for the 2-twist spun trefoil 
in  a closed braid form $\hat S$ with  two orientations 
 for $X={\sf R}_3$,  the dihedral quandle of order 3, and for a 3-cocycle
$\theta \in Z^3(X, \Z_3)$  specified in \cite{CJKLS}.
 We examine the invariant for one of the two orientations in Figure 19 of \cite{CJKLS}. 
%Let $\theta \in Z^3(X, \Z_3)$ be the quandle 3-cocycle defined by 
%the characteistic function 
%$$\theta = - \chi_{(0,1,0)} + \chi_{(0,2,0)} - \chi_{(0,2,1)} + \chi_{(1,0,1)} + \chi_{(1,0,2)} + \chi_{(2,0,2)} + \chi_{(2,1,2)} $$
%as  in \cite{CJKLS}. 
%In this case, the fundamental quandle has two generators, and any choice of elements of these generators give rise to a coloring of $\hat S$ with $X$. 
%Let $t$ be a multiplicative generator of the coefficient group $\Z_3$, where the cocycle takes values in multiplicative notation $t^k$, $k \in \Z_3 = \{ 0,1,2\}$. 
% with which the quandle cocycle invariant is  written in a polynomial form.  % Then the invariant value is $\Psi_X(\hat S)=3+6t$. 

The Hurwitz braid system of the surface braid is given by $\vec b = (b_1, \ldots, b_6)$, where
\begin{eqnarray*}
		b_1 &=& \sigma_3^{-1}\sigma_1 \sigma_2 \sigma_1^{-1}\sigma_3,\\
		b_2 &=& \sigma_3^{-1}\sigma_1^{-1}\sigma_2\sigma_1^{-1}\sigma_2^{-1}\sigma_1\sigma_3,\\
		b_3 &=& \sigma_3^{-1}\sigma_1^{-1}\sigma_2\sigma_3\sigma_2^{-1}\sigma_1\sigma_3,\\
		b_4 &=& \sigma_1^{-1}\sigma_2^2\sigma_3^{-1}\sigma_2^{-2}\sigma_1,\\
		b_5 &=& \sigma_1^{-1}\sigma_2^2\sigma_1\sigma_2^{-2}\sigma_1,\\
		b_6 &=& \sigma_2^{-1}.
\end{eqnarray*}
As discussed in Section~\ref{sec:quandle_col}, we have that the quandle $R_3$ induces a permutation representation of the braid group, which we indicate by $\rho$, and $R^\pm_i = \rho(\sigma_i^\pm)$. We will write $B_i = \rho(b_i)$, for $i=1,\ldots, 6$. 

MIVs correspond bijectively  to 
 colorings, and therefore each term in the sum of the cocycle invariant has a corresponding term in the image of $\Psi(S,\psi)$. However, as seen in Section~\ref{sec:quandle_col}, we know that $E(\vec B,1)$ generally contains a basis with more terms than MIVs. These are obtained considering the orbits of the action generated by the permutations associated to $\rho(b_1) = B_1, \ldots, \rho(b_6) = B_6$. A direct computation shows that $R_i^3 = \mathbb 1$, due to the coefficients being  taken modulo $3$. Then, we have
\begin{eqnarray*}
	B_1 B_2 
	&=& 	R_3^{-1}R_1R_2R_1^{-1}R_3R_3^{-1}R_1^{-1}R_2R_1^{-1}R_2^{-1}R_1R_3\\
	&=&		R_3^{-1}R_1R_2R_1R_2R_1^{-1}R_2^{-1}R_1R_3\\
	&=& 	R_3^{-1}R_1R_1R_2R_1R_1^{-1}R_2^{-1}R_1R_3\\
	&=&     R_3^{-1}R_1^3R_3  \\
	&=& 	1,
\end{eqnarray*}
which shows that $B_1 = B_2^{-1}$. Analogously, we have also that $B_3 = B_4^{-1}$ and $B_5 = B_6^{-1}$. Therefore, to determine the braid system eigenspace, following Lemma~\ref{lem:color}, we 
may only consider the orbits of the action generated by $B_1, B_3$ and $B_6$.  
A computer aided computation yields  
 the orbits as follows. Setting $E_{x,y,z,w} := e_x\otimes e_y\otimes e_z\otimes e_w$, we have:
\begin{eqnarray*}
	B_1 (E_{x,y,z,w}) &=&  E_{w, y+x-w, z+x-w, 2w-x}, \\
	B_3 (E_{x,y,z,w}) &=& E_{w+z-y, y, z, 2w-x+z-y},\\
	B_6 (E_{x,y,z,w}) &=& E_{x,2y-x,y,w}. 
\end{eqnarray*}  
Each application of $B_1, B_2, B_3$ fixes the differences $w-x$ and $z-y$. Therefore, the permutation action of $\langle B_1, B_3, B_6\rangle$ has orbits contained in sets of type $\{ E_{x,y,s+y,r+x}\ |\ x,y,r,s\in \Z_3  \}$. When $s = r = 0$, we have that each $E_{x,y,s+y,r+x}$ is fixed by the $B_i$'s, and therefore we obtain an MIV. There are $9$ such MIVs, as $x$ and $y$ vary in $\Z_3$, and they correspond to the quandle colorings as discussed above. A direct computation shows that for each $(r,s) \neq (0,0)$ we obtain different orbits, and therefore the braid system eigenspace $E_{(\vec B,1)}$ is a $17$-dimensional space, generated by $E_{x,y,s+y,r+x}$, where $8$ NIVs are given by $(r,s) \neq (0,0)$.
 
 Let $t$ be a multiplicative generator of the coefficient group $\Z_3$, where the cocycle takes values in multiplicative notation $t^k$, $k \in \Z_3 = \{ 0,1,2\}$.  
 The cocycle invariant is defined in a state sum form $\sum_{C_i} \prod_{T_j} \theta( x_j, y_j, z_j)^{\epsilon (T_j)}$,
 where $(x_j, y_j, z_j)$ corresponds to  the triple of quandle elements determined by 
  a coloring  $C_i$ at 
  a triple point $T_j$, and $\epsilon$ denotes the sign at $T_j$.

 In \cite{CJKLS}, Table 1 for Theorem 11.4 lists the value of 3-cocycle evaluated for each coloring and for each white vertex. There are 9 colorings, $C_i$, $i=1, \ldots, 9$, and 6 white vertices, $T_j$, $j=1, \ldots, 6$. These corresponds to the $9$ MIVs computed above, as expected from Lemma~\ref{lem:color}. We notice that the convention used in the present article realtive to the positioning of $*$ is different with respect to the one of \cite{CJKLS}. To compare the invariant of \cite{CJKLS} and $\Psi(S, \psi)$ directly, we need to adjust one of the conventions. 
 Therefore, we recompute the cocycle invariant of Theorem~11.4 of \cite{CJKLS} using the convention used in this article.

 Each $C_j$ corresponds to an MIV  $v_{ij}$ in $E(\vec B,1)$, where the pair $i,j$ 
 are pairs of integers parametrizing the colorings used 
   in the first column of Table~1 in \cite{CJKLS}. A direct computation applying $\psi$ in Lemma~\ref{lem:q3cocy}
 shows that the contribution for each such vector is 
given by 
$$\Psi(S, \psi) (v_{01}) = 2, \ \Psi(S, \psi) (v_{02}) = 1, \ \Psi(S, \psi) (v_{20}) = 1, \ \Psi(S, \psi) (v_{21}) = 2$$
 and $\Psi(S, \psi) (v_{ij}) = 0$ for the other choices of $i,j$. The relation between $\Psi(S,\psi)$ and the quandle cocycle invariant is that $\Psi(S, \psi) (v_{ij}) \in \Z_3$ is equal to the exponent of the contribution of the coloring $(ij)$ to the cocycle invariant.
  Namely,  the contribution $t^k$ in $\Phi_X(\hat S)$ corresponds to $k \in \Z_3$, the value of 
 $\Psi(S,\psi)$  for the corresponding color vector. Furthermore, the current invariant contains additional information coming from NIVs.  
  We notice that this relation mimicks the analogous result for the quantum 2-cocycle invariant found in \cite{SZ-QC}. Therefore, the same computation also gives each Boltzmann weight for the cocycle invariant of $S$ for MIVs. 
}
\end{example}

\begin{remark}
	{\rm 
			We point out that since $\Psi(S,\psi)$ is a functional which is shown to be invariant by Theorem~\ref{thm:3cocy}, a comparison of $\Psi(S_1,\psi)$ and $\Psi(S_2,\psi)$ for two surface braids is only meaningful when the domains are equal, i.e. $E(\vec B_1,1) = E(\vec B_2,1)$ where $\vec B_i$ is a braid system for $S_i$. However, by Theorem~\ref{thm:sfc_braid} we know that $E(\vec B_1,1) \neq E(\vec B_2,1)$ automatically implies that $S_1$ and $S_2$ are not equivalent. Therefore, $\Psi(S,\psi)$ can be used to distinguish surface braids whose eigenspace invariants fail to distinguish. 
	}
\end{remark}

\section{Towards invariants for knotted surfaces by closure}

It is known that two surface braids represent equivalent knotted surface under closure if they are related by a finite sequence of conjugation and stabilization~\cite{Kamada:SB}, in a manner analogous to the classical case.
In this section we present an approach towards defining invariants for knotted surfaces from the invariants defined for surface braids in preceding sections, by examining effects under conjugation and stabilization. 
%We assume throughout the section that the ring $\mathbb k$ is a PID, and that $R : V\otimes V \longrightarrow V\otimes V$ is a YBO where $V$ is free over $\mathbb k$.  In this case $V$ is flat, and Lemma~\ref{lem:supset} implies that $E(\vec B, \lambda)= E_{(\vec B', \lambda)}$, so that we assume this equality throughout this section and do not make distinction. 
We recall the following notions from \cite{Kamada:SB}. 
	\begin{definition}
		{\rm 
			Let $\vec b = (b_1, \ldots, b_m)$ be a braid system of degree $n$, i.e. where each braid $b_i \in \mathbb B_n$. Then, we say that $\vec b' = (b'_1, \ldots, b'_m)$ is:
			\begin{itemize}
				\item 
				Obtained from $\vec b$ via {\it conjugation} if $b'_i = b^{-1}b_ib$ for each $i=1,\ldots , m$ and some $b\in \mathbb B_n$;
				\item 
				Obtained from $\vec b$ via {\it stabilization} if $\vec b' = (b_1, \ldots, b_m, \sigma_n, \sigma_n^{-1})$. % If the roles of $\vec b$ and $\vec b'$ are exchanged, one
				One also  says that 
				%a {\it destabilization} has been performed. 
				$\vec{b}$ is obtained from $\vec b' $ by {\it destabilization}. 
			\end{itemize}
		}
	\end{definition}
	
The changes on charts that correspond to conjugation and stabilization are depicted in Figure~\ref{stab}.
The figure (A) is the original chart, (B) represents a chart after a conjugation and (C) represents stabilization.

\begin{figure}[htb]
\begin{center}
\includegraphics[width=4in]{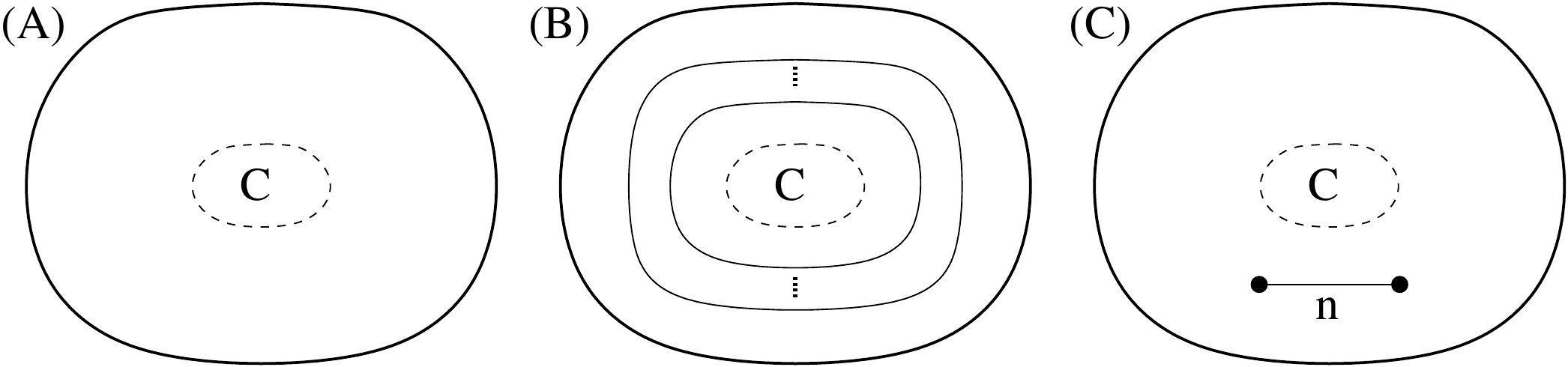}
\end{center}
\caption{Conjugation and stabilization of surface braids on charts}
\label{stab}
\end{figure}

	The Alexander theorem in dimension four guarantees that any knotted surface (in $4$-space) is
	 equivalent to 
	 the closure of a $2$-braid, represented via a braid system. 
	The closure is defined in a manner similar  to the classical case, as the operation of capping off the boundary of surface braids by trivial nested discs outside \cite{Kamada:SB}. 
	The Markov theorem in dimension four establishes when two surfaces represented by two braid systems are equivalent (ambient isotopic). Namely, this happens if and only if their braid systems are related by slide equivalences, conjugations and (de)stabilizations. 
	
	Given a braid system $b = (b_1, \ldots, b_m)$ of a surface braid $S$, and a YBO $R$ with an eigenvalue $\lambda$, we consider the braid system eigenspace $E(\vec B, \lambda)$ and  $E_{(\vec B', \lambda)}$. 
	%In the following, we will indicate the eigenspace of $\lambda =1$ simply by $E(\vec B)$. 

\begin{lemma}
	Let $S$ be a surface $n$-braid, and $R$, $\lambda$,   $E(\vec B, \lambda)$, $E_{(\vec B, \lambda)}$ be as above.
Suppose that $\vec b' = (b'_i)$ is obtained from $\vec b=(b_i)$ by a conjugation, $b_i'=b^{-1} b_i b$ for $i=1, \ldots, m$. Let $\vec{B}=\rho (\vec b)$  and $\vec{B'}=\rho (b')$.
Then %there exists an 
the isomorphism $Q=\rho(b) $ induced by a braid word $b \in \mathbb B_n$ 
satisfies that % such that 
$E(\vec B', \lambda)=Q^{-1} E(\vec B, \lambda)$ and $E_{(\vec B', \lambda)} =Q^{-1} E_{(\vec B, \lambda)}$.
\end{lemma}

\begin{proof}
Under the assumption we have  $\vec B' = (Q^{-1}B_1Q, \ldots,  Q^{-1} B_mQ)$. 
For each $i=1, \ldots, m$ we have $E(B'_i, \lambda) = Q^{-1} (E(B_i, \lambda))$. 
		Since $Q=\rho (b)$ is an isomorphism, it follows that $$\cap_{i=1}^nQ^{-1} (E(B_i, \lambda) = Q^{-1} (\cap_{i=1}^n(E(B_i, \lambda))), $$  from which it follows that $E(\vec B', \lambda) = Q^{-1} E(B, \lambda)$. 
		The other case is similar. 
		\end{proof}

\begin{corollary}
The rank, ${\rm rank} E(\vec B, \lambda)$, is an invariant under conjugation of surface braids.
\end{corollary}

\begin{remark}
{\rm
If a given YBO $R$ preserves a specific structure, such as being an orthogonal matrix, 
then the subspaces related by compositions $\rho(b)$ of such matrices may have further obstructions than their ranks, and the rank alone may not characterize equivalence under conjugation.
}
\end{remark}

\begin{remark}\label{rem:stab}
	{\rm As for stabilization, we observe the following.
		The surface braid $S$ of index 3 in Example~\ref{ex:22} with braid system $\vec{b}=(\sigma_2, \sigma_2^{-1})$ is a stabilization of the trivial surface braid  of index 2. From Example~\ref{ex:22}, with the representation 
		$\rho$ using the quandle ring $\Z [{\sf R}_3]$ of the order 3 dihedral quandle ${\sf R}_3$, we have 
		${\rm rank}E(\vec B, 1) = 15$, where $\vec B=\rho (\vec b)$.
		Since the braid system of $S_0$ is empty, we have $\vec B=I$, the identity matrix of size 2, and
		${\rm rank}E(I, 1) = 9$. 
		Hence the rank of $E(\vec B, \lambda)$ is not invariant under stabilization.
	}
\end{remark}

%From the definitions we also have the following.

We assume throughout the rest of the section that the ring $\mathbb k$ is a PID, and that $R : V\otimes V \longrightarrow V\otimes V$ is a YBO where $V$ is free over $\mathbb k$.  In this case $V$ is flat, and Lemma~\ref{lem:supset} implies that $E(\vec B, \lambda)= E_{(\vec B', \lambda)}$, so that we assume this equality throughout this section and do not make distinction.

\begin{lemma}\label{lem:ses}
	Let $S$ be a surface $n$-braid, and $R$, $\lambda$,   $E(\vec B, \lambda)$ be as above.
Suppose that $\vec b' = (b'_i)$, $i=1, \ldots, m+2$, be obtained from $\vec b=(b_i)$, $i=1,\ldots, m$ by a stabilization, $b_i'= b_i$ for $i=1, \ldots, m$ and $b'_{i+1}=\sigma_n$, $b'_{i+2}=\sigma_n^{-1}$.
Let $\vec{B}=\rho (b)$  and $\vec{B'}=\rho (b')$. Then we have 
\begin{eqnarray*}
E(\vec B', \lambda) &=&  E(\vec B, \lambda) \otimes V \cap 
[V^{\otimes (n-1)}\otimes (\ker (\mathbb 1-R))].
\end{eqnarray*}
\end{lemma}
\begin{proof}
From the proof of Lemma~\ref{lem:supset},
we have 
$$E(\vec{B},\lambda)=\bigcap_{i=1}^{m}E(B_i\otimes \mathbb 1,\lambda)  = (\cap_{i=1}^{m}E(B_i,\lambda))\otimes V = E(\vec B,\lambda)\otimes V. $$
Since $$E(B_{m+1}, \lambda)=E(B_{m+2}, \lambda)=
E(R_{(n)}, \lambda)=
[V^{\otimes (n-1)}\otimes (\ker (\mathbb 1-R))],$$
we obtain
$$ E(\vec{B'} , \lambda) = \bigcap_{i=1}^{m+2}E(B_i\otimes \mathbb 1,\lambda) =E(\vec B,\lambda)\otimes V
\cap
[V^{\otimes (n-1)}\otimes (\ker (\mathbb 1-R))]$$
as desired.
\end{proof}

Towards constructing an invariant of knotted surfaces, we consider a restricted family of YBOs as follows. We let $V$ be freely generated by the finitely many elements $\{e_i\}_{i=1}^n$, and we assume  $\mathbb k$ to have zero characteristic. 
Observations in  Examples~\ref{ex:22} and Remark~\ref{rem:stab} lead to considering the following special type of YBOs to seek properties that remain invariant  under stabilization. The changes under stabilization caused in these examples are due to multiple tensor terms that appear in the eigenspaces. Thus we consider the following restrictions.

\begin{definition}
{\rm 
We say that a YBO  $R$ on $V$ is \emph{simple}, for some $k\in \N$, if there is a basis $\{e_i\}_{i=1}^n$ of $V$ such that $R$ maps simple tensors to simple tensors.
% and if for each generator $e_i$ there exist $k$ generators $e_{s(i,j)}$, $j=1, \ldots, k$, such that $R(e_i\otimes e_{s(i,j)}) = e_i\otimes e_{s(i,j)}$ for each $j=1,\ldots, k$. 

}
\end{definition}

In this case we have $R(e_i\otimes e_j) = R_{i,j}^{p,q}e_p\otimes e_q$, where 
there is a single term on the right-hand side and no sum involved.

\begin{definition}
{\rm
We say that a YBO  $R: V \otimes V \rightarrow V \otimes V $ is \emph{$k$-stabilizing}, for some $k\in \N$,
if for each basis vector $e_i$ there exist $k$ basis vectors  $e_{s(i,j)}$, $j=1, \ldots, k$, such that $R(e_i\otimes e_{s(i,j)}) = e_i\otimes e_{s(i,j)}$ for each $j=1,\ldots, k$. 

Similarly, a set-theoretic YBO (SYBO) $\bar R: X \times X \rightarrow X \times X$ is called  \emph{$k$-stabilizing} if 
for each $x \in X$,  there exist $k$ elements   $s(x,j)  \in X$ such that $\bar R(x, s(x,j)) = (x, s(x,j) ) $ for each $j=1,\ldots, k$. 

}
\end{definition}

\begin{definition}
	{\rm 
			Let $\bar R: X\times X\longrightarrow X\times X$
			be an SYBO  given by $\bar R(x,y) = (R_1(x,y),R_2(x,y))$. 
			A
			\emph{YB 2-cocycle} 
			 is a map $\sigma : X\times X \longrightarrow \mathbb k$ satisfying the following equation:
			\begin{eqnarray*}
& & 	\sigma(x,y)\sigma(R_2(x,y),z)\sigma(R_1(x,y),R_1(R_2(x,y),z)) 
		\\
		 &=& \sigma(y,z)\sigma(x,R_1(y,z))\sigma(R_2(x,R_1(y,z)),R_2(y,z)).  
			\end{eqnarray*}
	}
\end{definition}

\begin{remark}
{\rm 
The preceding condition  is  the 2-cocycle condition for SYBOs defined and studied in 
\cite{CES:SYBO} with values in the multiplicative group of $\mathbb k$ in multiplicative notation.  
In \cite{CES:SYBO}, deformations of SYBOs by YB 2-cocycles are defined. The YBO $ R$ described below in Lemma~\ref{lem:admissible} is the linearlization of such  deformations by YB 2-cocycles. 
}
\end{remark}

A direct verification yields the following. 

\begin{lemma}\label{lem:admissible}
			Let $X$ be a set 
						and let $V$ denote the free module over $\mathbb k$ generated by symbols $e_x$ for $x\in X$. Then, the YBO $ R : V\otimes V \rightarrow V\otimes V$ 
					 is simple and $k$-stabilizing  if and only if it is given by 
$$ R(e_x \otimes e_y) = \sigma(x,y)\cdot e_{R_1(x,y)}\otimes e_{R_2(x,y)}$$ for some  $k$-stabilizing SYBO $\bar R: X \times X \rightarrow X \times X $, $\bar R(x,y)=(R_1(x,y), R_2(x,y) ) $, and a YB $2$-cocycle $\sigma$.
					 
\end{lemma}

\begin{definition}
	{\rm 
			Let $\hat S$ denote a knotted surface obtained as the closure of a 
			surface braid $S$ represented by the braid system $\vec b = (b_1, \ldots, b_m)$. Let $R$ be 
			as in Lemma~\ref{lem:admissible}. 
			 Let $E(\vec B,1)$ denote the braid system $1$-eigenspace associated to $\vec b$ and $\rho (R)$ for a braid group representation $\rho$. 
		Denote by 	$\Sigma(V^{\otimes n})$  the set of vectors of the form $e_{i_1}\otimes \cdots \otimes e_{i_n}$. 
Then, we define the \emph{simple rank} of $\hat S$ as 
			$$ {\rm rank}_{\Sigma} (\hat S) = \frac{1}{k^n} \sharp(\Sigma(V^{\otimes n})\cap E(\vec B,1)). $$
			%where $\Sigma(V^{\otimes n})$ denotes the set of vectors of the form $e_{i_1}\otimes \cdots \otimes e_{i_n}$. 
	}
\end{definition} 

\begin{theorem}
			The simple rank of $\hat S$ is an invariant of knotted surfaces. 
\end{theorem}
\begin{proof}
			To show that $\rank_\Sigma(\hat S)$ is a well defined invariant of knotted surfaces, it is sufficient to show its invariance under slide equivalence, conjugation and (de)stabilization with respect to the braid system.  By Lemma~\ref{lem:admissible}, $R$ is given by a set-theoretic YBO $\bar R$,  and a YB $2$-cocycle $\sigma$ satisfying the stated properties. 
			
			Suppose that $\vec b' = (b_1, \ldots, b_ib_{i+1} b_i^{-1}, b_i, \ldots, b_m)$ is obtained from $\vec b$ by applying a slide move. As shown in Theorem~\ref{thm:sfc_braid}, we have the equality $E(\vec B',1) = E(\vec B,1)$, and therefore in particular the equality $\Sigma(V^{\otimes n})\cap E(\vec B',1) = \Sigma(V^{\otimes n})\cap E(\vec B,1)$ holds. 
			
			Let us consider the case where $\vec b'$ is obtained from $\vec b$ by a stabilization, which in turn means that $\vec b$ is obtained from $\vec b'$ by a destabilization. By Lemma~\ref{lem:ses}, we have that 
			$$E(\vec B',1) = [E(\vec B, 1)\otimes V] \cap [V^{\otimes (n-1)}\otimes (\ker (\mathbb 1-R))].$$
			Let $e_{x_1}\otimes \cdots \otimes e_{x_n}$ denote an element of $\Sigma(V^{\otimes n})\cap E(\vec b,1)$. Since $R$ is $k$-stabilizing, there exist $s(x_n,j)\in X$, for $j=1, \ldots, k$, such that $e_{x_1}\otimes \cdots \otimes  e_{x_n}\otimes e_{s(x_n,j)}$ is in  $\Sigma(V^{\otimes {n+1}})\cap E(\vec B',1)$. Conversely, if $e_{x_1}\otimes \cdots \otimes  e_{x_n}\otimes e_{x_{n+1}}$ is an element of $\Sigma(V^{\otimes {n+1}})\cap E(\vec B,1)$, from the same equality  
			 it follows that $e_{x_n}\otimes e_{x_{n+1}}$ is a solution to the equation $R(e_{x_n}\otimes e_y) = e_{x_n}\otimes e_y$ which gives $x_{n+1} = s(x_n,j)$ for some $j$, since $R$ is $k$-stabilizing. Therefore the number of elements in $\Sigma(V^{\otimes n})\cap E(\vec B,1)$ and
			 that in  $\Sigma(V^{\otimes {n+1}})\cap E(\vec B',1)$ are proportional by $k$,
              $$
              \sharp \Sigma(V^{\otimes {n+1}})\cap E(\vec B',1)
             = k \ \sharp \Sigma(V^{\otimes {n+1}})\cap E(\vec B,1).
             $$
             Since $\rank_\Sigma(\hat S)$ is normalized by $k^n$ and $n$ increases by $1$, it follows that the simple rank is invariant:
        	$${\rm rank}_\Sigma (\hat S)=
             \frac{1}{k^{n+1}}\sharp (\Sigma(V^{\otimes {n+1}})\cap E(\vec B',1))=\frac{1}{k^{n+1}}\sharp (\Sigma(V^{\otimes {n}})\cap E(\vec B,1))
             ={\rm rank}_\Sigma ( S)
             $$
             as desired.
			
			Let $\vec b'$ be obtained from $\vec b$ by a conjugation as $\vec b' = (\beta^{-1} b_1\beta,\ldots, \beta^{-1} b_n\beta)$, where $\beta$ is an $m$-braid. Then, applying Lemma~\ref{lem:conjeigen} we have $E(\vec b',1) = \rho(\beta^{-1})(E(\vec b,1))$. 
            
            Since $R$ maps simple tensors to simple tensors, so does $\rho (\beta)$. In particular, we have 
            $$\rho (\beta^{-1})(e_{x_1}\otimes \cdots \otimes e_{x_n}) = \alpha(x_1,\ldots, x_n) e_{y_1}\otimes \cdots \otimes e_{y_n}$$ for some nonzero $\alpha(x_1,\ldots, x_n)\in \mathbb k$ determined by the YB 2-cocycle %YB-twist 
			$\sigma$, and some $y_1, \ldots, y_n\in X$. Therefore we have a $1-1$ correspondence between  elements of $\Sigma(V^{\otimes n})\cap E(\vec B,1)$ and elements of $\Sigma(V^{\otimes n})\cap E(\vec B',1)$. This completes the proof. 
\end{proof}

\begin{remark} 
{\rm 
The definition of the simple rank mimics the coloring invariant of discrete structures such as biquandles and their variants. This attempt is to resolve the difference that occurs under stabilization to the number of simple tensors. Potential invariants beyond discrete cases may be expected, for example, from the construction of YBOs from linearlized ones by YB $2$-cocycles. 
Computations of Yang-Baxter $2$-cocycles, however,  have not been fully developed, and are desirable. 
}
\end{remark}

In this paper we defined two types of invariants for surface braids. Charts that describe surface braids are planar graphs, with two types of vertices corresponding to singular points of projections; black vertices representing branch points and white vertices for triple points. 
Black vertices give rise to monodromies called braid systems. 
From a braid system  that describes a given surface braid, we associated the intersection of eigenspaces of a representation of the braid words obtained from a YBO. 
It was shown that the invariant distinguishes infinitely many surface braids. 
For white vertices representing triple points, we used YB 3-cocycles from the Yang-Baxter cohomology theory, in an analogy with quandle 3-cocycle invariants, 
 to construct a quantum 3-cocycle invariant. Unique color vectors of source regions of white vertices are used for a given 1-eigenvector at a base points, and evaluated by a given YB 3-cocycle. 
 For the eigenspace invariant, we provided an approach towards constructing knotted surface invariants by closure, investigating effects under conjugations and stabilization. Similar analyses for invariance under conjugation and stabilization for the quantum 3-cocycle invariant  remain elusive. 
Further computations, properties, applications and relations to other invariants would be desirable.

   \bigskip        	
%

%\noindent
%{\bf Acknowledgements.} 

\noindent
{\bf AI use declaration}
During the preparation of this work, the authors used some AI tools to assist in literature discovery and background concept verification. No text, figures, or mathematical derivations in the final manuscript were generated by AI. The human authors reviewed and verified all information and retain sole responsibility for the accuracy and integrity of the work.

\end{document}